\documentclass[10pt,twoside]{amsart}
\usepackage{amssymb, amsmath, mathrsfs, mathtools, amsfonts, amsthm, amscd, yfonts}
\usepackage{xcolor}
\usepackage{float}
\usepackage{hyperref} 
\usepackage{float}
\usepackage{multicol}
\usepackage{tikz-cd}

\theoremstyle{plain}
\newtheorem{Theorem}{Theorem}
\newtheorem{Proposition}[Theorem]{Proposition}
\newtheorem{Lemma}[Theorem]{Lemma}
\newtheorem{Corollary}[Theorem]{Corollary}
\newtheorem{Conjecture}[Theorem]{Conjecture}
\theoremstyle{definition}
\newtheorem{Definition}[Theorem]{Definition}
\theoremstyle{remark}
\newtheorem{Remark}[Theorem]{Remark}
\newtheorem{Example}[Theorem]{Example}

\newcommand{\rR}{\mathscr{R}}

\title{On regular quantum commutative algebras}

\author[Y. Bahturin]{Yuri Bahturin$^{1,2}$}
\address{1. Department of Mathematics and Statistics, Memorial University of Newfoundland, St. John's, NL, Canada, A1C 5S7}
\email{bahturin@mun.ca}
\author[L. Centrone]{Lucio Centrone$^2$}
\address{2. Dipartimento di Matematica, Universit\`a degli Studi di Bari Aldo Moro, Via Edoardo Orabona, 4, 70125 Bari, Italy}
\email{lucio.centrone@uniba.it}
\author[K. Pereira]{Kau\^e Pereira$^3$}
\address{3. IMECC, UNICAMP, Rua S\'ergio Buarque de Holanda 651, 13083-859 Campinas, SP, Brazil}
\email{k200608@dac.unicamp.br}
\thanks{K. Pereira was supported by FAPESP Grant 2023/01673-0, and FAPESP Grant 2025/03763-2}

\subjclass[2020]{16R10, 16R50, 16W55, 16T05}

\keywords{Regular gradings; Regular quantum commutative algebras; PI algebras}

\begin{document}
\begin{abstract}
Let $K$ be an algebraically closed field of characteristic different from $2$. We provide a positive solution to the Bahturin--Regev conjecture in the general finite-dimensional (non-graded) setting, assuming that $\operatorname{char}(K)$ does not divide the quantum length of a minimal regular quantum commutative decomposition. Furthermore, we obtain a criterion, formulated in terms of regular quantum commutative decompositions, under which a set-grading on a semisimple associative algebra is realized as a group grading.
\end{abstract}
\maketitle

\section{Introduction}\label{s1}
The theory of associative algebras satisfying polynomial identities has attracted sustained attention over the past decades. A major breakthrough in the area was achieved by Kemer in the period 1984--1986, whose work remains a cornerstone in the structure theory of PI-algebras. 

In \cite{Kemer.book}, Kemer introduced the notion of a $T$-prime ideal. A $T$-ideal $P$ of $K\langle X\rangle$ (the free associative algebra freely generated by a countable infinite set $X$) is said to be $T$-\emph{prime} if, whenever $U_{1}U_{2}\subseteq P$ for some $T$-ideals $U_{1}$ and $U_{2}$, it follows that $U_{1}\subseteq P$ or $U_{2}\subseteq P$. 

Kemer proved that amy $T$-prime ideal over the field of characteristic zero is an ideal of identities of an algebra of one of three types
\[
M_{n}(K),\,M_{n}(E),\,M_{p,q}(E),
\]
where $E$ is the Grassmann algebra of countable dimension  and $M_{p,q}(E)$ is the Grassman envelope of the associative superalgebra $M_{p.q}(K)$ (see \cite[Theorem 8.2.2]{Drensky.Formanek.Book}). One says that  $M_{n}(K)$, $M_{n}(E)$, and $M_{p,q}(E)$ are $T$-\emph{prime algebras}. 
Each of these algebras admits a natural $\mathbb{Z}_{2}$-grading. In \cite{Kemer.book}, Kemer  provided a list of PI-equivalences among such algebras and showed that showed  that the tensor product of $T$-prime algebras is again $T$-prime (see \cite[Theorem 6.1]{regevz2}). 

A natural generalization of the latter result deals with so called \emph{twisted tensor products}. Given two $\mathbb{Z}_{2}$-graded algebras $A$ and $B$, their twisted or $\mathbb{Z}_{2}$-graded tensor product, denoted by $A\overline{\otimes}B$, is  an algebra with underlying vector space $A\otimes B$. Given homogeneous  elements $a,c\in A$ and $b,d\in B$ their product is defined by
\[
(a\overline{\otimes} b)(c\overline{\otimes} d)
:=(-1)^{\deg_{\mathbb{Z}_{2}}(b)\deg_{\mathbb{Z}_{2}}(c)}(ac\overline{\otimes} bd),
\]
where $\deg_{\mathbb{Z}_{2}}(\cdot)$ denotes the degree with respect to the grading on each algebra.

It was conjectured in \cite{regevz2} that if $A$ and $B$ are $T$-prime algebras, then $A\overline{\otimes}B$ is again $T$-prime. This conjecture was partially verified in \cite{regevz2}, and later completely solved by Freitas and Koshlukov in \cite{Freitas.Koshlukov}, and independently by Di Vincenzo and Nardozza in \cite{DiVincenzo.Nardozza} . 

The main technique underlying these developments, introduced by Regev and Seeman in \cite{regevz2}, consists in establishing a matrix criterion to detect PI-equivalence. In this context arises the notion of a regular algebra, which in this paper we refer to as \emph{regular quantum commutative algebra}. 

Let $R$ be an associative algebra over a field $K$ with $\operatorname{char}(K)\neq 2$, and let \begin{equation}\label{eRQCA} 
\rR:\:R=R_{1}\oplus \cdots \oplus R_{m}
\end{equation}
 be a vector space decompositon  of $R$. We say that (\ref{eRQCA}) is a \emph{regular quantum commutative decomposition} of $R$, and that $R$ is a \emph{regular quantum commutative algebra}, if the following conditions hold: 

(1) For every $1\leq i,j\leq m$, there exists a scalar $\theta(i,j)\in K^{\ast}$, depending only on the components $R_{i}$ and $R_{j}$, such that $a_{i}a_{j}=\theta(i,j)a_{j}a_{i}$ for all $a_{i}\in R_{i}$ and $a_{j}\in R_{j}$; 

(2) For every $n\in \mathbb{N}$ and every $(i_{1},\ldots,i_{n})\in\{1,\ldots,m\}^{n}$, there exist homogeneous elements $a_{1}\in R_{i_{1}},\ldots,a_{n}\in R_{i_{n}}$ such that $a_{1}\cdots a_{n}\neq 0$. 

In this setting, $m$ is called the \emph{quantum length} of the regular quantum commutative decomposition (\ref{eRQCA}), the scalars $\theta(i,j)\in K^{\ast}$ define the \emph{regular quantum commutation function} $\theta$ of $R$. One associates with (\ref{eRQCA}) the \emph{regular quantum decomposition matrix} 
\begin{equation}\label{eRQCDM}
 M^{R}=(\theta(i,j))\mbox{ where }1\leq i,j\leq m.
\end{equation} As shown in \cite[Theorem 3.1]{regevz2}, if $R$ and $A$ are algebras with such decompositions having the same number of components, and if $M^{R}$ is obtained from $M^{A}$ by a permutation of rows, then $R$ and $A$ satisfy the same multilinear identities. In particular, if $\operatorname{char}(K)=0$, then $R$ and $A$ are PI-equivalent. 

Later, Bahturin and Regev introduced in \cite{bahturin2009graded} the notion of a regular grading by an abelian group. These structures correspond to regular quantum commutative algebras in which the decomposition coincides with the grading, and the quantum commutation function becomes a skew-symmetric bicharacter of the group. In this setting, the values of the bicharacter are roots of unity, and the non-degeneracy condition is equivalent to the absence of nontrivial graded \emph{monomial} identities.

Such structures arise naturally in the theory of graded algebras. Notable examples include the natural $\mathbb{Z}_{2}$-grading on the Grassmann algebra $E$, the Sylvester grading on $M_{n}(K)$, and the natural $G$0-grading on a twisted group algebra of a group $G$. These algebras  proved to be fundamental in the study of \emph{twisted tensor products} (see \cite[Definition 3.1]{bahturin2009graded}). Twisted tensor products generalize the $\mathbb{Z}_{2}$-graded tensor products and appear in several areas, including geometry and mathematical physics. 

In \cite{bahturin2009graded}, the authors introduced the notion of minimality for such decompositions. A decomposition (\ref{eRQCA}) is called \emph{minimal} if no two rows (or equivalently, columns) of $M^{R}$ coincide. Equivalently, the decomposition cannot be ``coarsened'' by merging two or more components. In the same work, the authors conjectured that (\ref{eRQCA}) is minimal if and only if $\det M^{R}\neq 0$. Moreover, in this case, both the number of direct summands and the determinant of $M^{R}$ are independent of the chosen decomposition. 

We observe that, in the context of regular gradings, minimality is equivalent to the non-degeneracy of the corresponding bicharacter. Moreover, it is immediate that $\det M^{R}\neq 0$ implies minimality. As an illustration, consider the matrix algebra $M_{n}(K)$ endowed with its Sylvester grading by the group $\mathbb{Z}_{n}\times \mathbb{Z}_{n}$. It was shown by Bahturin, Regev, and Zeilberger in \cite{bcommutation} that, over an algebraically closed field of characteristic $0$, one has $\det M^{M_{n}(K)}=\pm n^{n^{2}}$. In particular, this grading is minimal, and hence yields a minimal regular decomposition of $M_{n}(K)$.

In \cite{Eli1}, Aljadeff and David gave a positive answer to the Bahturin--Regev conjecture in the context of regular gradings, assuming that $K$ is an algebraically closed field of characteristic $0$. They extended the notion of a regular grading to the gradings by arbitrary finite groups and showed that the conjecture holds in this more general setting. More precisely, if $R=\bigoplus_{g\in G}R_{g}$ is a regular grading with minimal decomposition, then $\det M^{R}=\pm |G|^{|G|/2}$, and the order of the group $|G|$ coincides with $\exp(R)$, the PI-exponent of $R$. Moreover, the authors provided a complete description of the graded identities of such algebras and derived consequences for the associated decomposition matrices. Remarkably, the minimal polynomials of these matrices lie in $\mathbb{Z}[x]$ (see \cite[Corollary 47]{Eli1}).  We also mention that, in this work, the authors employ powerful cohomological tools to show the folowing.  Given a regular grading by a (not necessarily abelian) group $G$ with the commutation function $\theta$ satisfying $\theta(g,g)=1$ for all $g\in G$, the function $\theta$ is induced by a $2$-cocycle $\alpha\in Z^{2}(G,K^{\ast})$. This result plays a fundamental role in the study of twisted group algebras. For further details on the use of twisted group algebras in other contexts, for instance in the study of generic $G$-graded Azumaya algebras introduced by Aljadeff and Karasik in \cite{Aljadeff.Karasik}, we refer the reader to \cite{Aljadeff.Haile.Karasik}.

Bahturin - Regev's conjecture fails in general for regular gradings over fields of characteristic strictly greater than $2$. 

Given a finite abelian group $G$ and a decomposition $R=\bigoplus_{g\in G}R_{g}$ as a $G$-graded algebra satisfying the first of the two regularity conditions with respect to a bicharacter $\theta\colon G\times G\rightarrow K^{\ast}$, one calls $R$  a \emph{$\theta$-commutative algebra}. Such algebras form a subclass of \emph{$\theta$-solvable} and \emph{$\theta$-nilpotent} algebras, which have been extensively studied by Bahturin, Montgomery, and Zaicev in \cite{Bahturin.Montgomery.Zaicev}, as well as by Bahturin and Parmenter in \cite{Bahturin.Parmenter} and other authors. In a more general context of algebras with the action of Hopf the generlized commutative algebras has been studied in Bahturin, Fischman, and Montgomery in \cite{Bahturin.Fischman.Montgomery}, where the authors prove an analogue of the classical theorem of Scheunert \cite{M. Scheunert} in the setting of commutative and cocommutative Hopf algebras.

Other works in the context of regular gradings include \cite{L.P.K.2}, where finite-dimensional regular gradings by finite abelian groups are described, \cite{L.P.K.3}, where infinite-dimensional $\mathbb{Z}_{2}$-graded regular algebras are studied, and \cite{L.P.C.K.1}, where the authors apply regular gradings to the study of the primeness property for graded central polynomials.

The main objective of this paper is to investigate regular quantum commutative decompositions in the general (non-graded) setting and to establish criteria under which set-gradings on semisimple associative algebras can be realized as group gradings. The results obtained extend and unify several aspects of the theory of regular decompositions, providing new structural insights beyond the graded framework. In what follows, we provide a brief overview of the main results of each section of the paper. Throughout this paper, we assume that $K$ is an algebraically closed field of characteristic different from $2$. In some of the main results, additional restrictions on the characteristic of the field are required.

 In Section \ref{s2}, we provide a brief review of some notions of the Cohomology Theory, Graded Algebras and PI-theory.
 
	In Section \ref{s3}, we consider regular gradings by arbitrary finite groups, as defined in \cite{Eli1}. The main results are Propositions \ref{infinite.dimensional} and \ref{Proposition.KG}. The first proposition  provides a necessary condition under which a regular $G$-graded algebra is finite-dimensional. In the second, it is shown that if $R$ is a regular $G$-graded algebra, then it contains a copy of the twisted group algebra $K^{\alpha}G$, for a 2-cocycle $\alpha$. This result is analogous to \cite[Proposition 26]{L.P.K.2} and \cite[Theorem 17]{L.P.K.3}.

	The main goal of Section \ref{s4} is to study the regular quantum commutation function. In the case of regular group gradings, this function becomes a bicharacter of the underlying group, hence its values are roots of unity. So it is natural to investigate if the same is true in the case of general regular quantum commutative algebras.
	
	 Using the techniques developed by Regev in \cite{Regev1972}, we show in Proposition \ref{commutattive.PI} that regular quantum commutative algebras are PI-algebras. Applying methods of PI-theory, we conclude in Theorem \ref{nth.root} that the values of a regular quantum commutation function are indeed roots of unity in the case where $\operatorname{char}(K)=0$. 
	
	In the case of finite-dimensional algebras, we prove a characteristic-free Theorem \ref{general.finite.dimensional.case} saying that the values of the quantum commutration function are always roots of unity. 
	
	 In Section \ref{s5}, we give a positive answer to  Bahturin--Regev's  conjecture for a minimal finite-dimensional regular quantum commutative algebra $R$ of quantum length $m$, in the case where $\operatorname{char}(K)\nmid m$. Namely, we prove in Theorem \ref{Bahturin.Regev}, that under the above restrictions, $\det M^{R}=\pm m^{m/2}$ and $m=\dim(C)$, where $C$ is a simple subalgebra of $R$. 
	
	In particular, if $\operatorname{char}(K)=0$, then $m$ coincides with the PI-exponent of $R$ (see, for instance, \cite{Eli1}).
 
	The main objective of Section \ref{s6} is the study of the connection between regular quantum commutative decompositions and the gradings of an algebra. We begin by considering the matrix algebra $M_{n}(K)$. In Theorem \ref{group.grading.matrices.algebra} we show that any decomposition of the form 
	\[
	M_{n}(K)=L_{1}\oplus \cdots \oplus L_{m},
	\]
satisfying the first regularity condition, where each $L_{i}$ is a nonzero subspace, is a regular grading by a finite abelian group $G$. In Theorem \ref{Set.grading.semissimple} we show that if $\mathscr{R}$ is a set-grading on a finite-dimensional semisimple algebra $R$ such that $\rR$ is a minimal regular quantum commutative decomposition, then $\rR$ must be a group grading by a finite abelian group.
	
	The main result of Section \ref{s7}\, is Proposition \ref{necessary.condition}\, giving a necessary condition for a finite-dimensional algebra to admit a regular quantum commutative decomposition. This result is further applied to describe the structure of regular quantum commutative algebras. 	For instance, when $\operatorname{char}(K)=0$, we provide an example of an algebra that does not admit any regular quantum commutative decomposition, but does admit a regular grading by a non-abelian group, in the sense of Aljadeff and David \cite{Eli1}.
	
	Finally, it is shown that if $G$ is a group (not necessarily abelian) of order $p^{n}$, where $p$ is a prime number, and $\operatorname{char}(K)\nmid |G|$, then the group algebra $KG$ admits the structure of a regular grading by a finite abelian group. 

\section{Background}\label{s2}

Throughout this paper, $K$ denotes an algebraically closed field with $\operatorname{char}(K)\neq 2$. Let $G$ be a group with the neutral element $e\in G$. We assume that $G$ acts trivially on $K$. A function $\alpha\colon G\times G\rightarrow K^{\ast}$ satisfying
 \begin{equation}\label{e1}
\alpha(g,h)\alpha(gh,k)=\alpha(g,hk)\alpha(h,k),\mbox{ for all } g,h,k\in G,
\end{equation}
is called a \textit{2-cocycle} on $G$; the set $Z^{2}(G,K^{\ast})$ of all 2-cocycles on $G$ is an abelian group under the pointwise multiplication. A 2-cocyle $\delta$ is called a \textit{coboundary} on $G$ if there exists a function $\eta\colon G\rightarrow K^{\ast}$ satisfying $\eta(e)=1$ such that
\begin{equation}\label{e2}
\delta(g,h)=\eta(g)\eta(hy)\eta(gh)^{-1}\mbox{ for all } g,h\in G.
\end{equation}
 The subgroup of all coboundaries on $G$ is denoted by $B^{2}(G,K^{\ast})$. The second cohomology group of $G$ is the quotient group $H^{2}(G,K^{\ast})= Z^{2}(G,K^{\ast})/B^{2}(G,K^{\ast})$. The equivalence class of $\alpha\in Z^{2}(G,K^{\ast})$ in $H^{2}(G,K^{\ast})$ is denoted by $[\alpha]$. 
 A particular type of $2$-cocycles are \emph{bicharacters}. A bicharacter $\beta:G\times G\to K^*$ is a function satisfying
 \begin{equation}\label{e3}
 \beta(gh,k)=\beta(g,k)\beta(h,kz)\mbox{ and }\beta(g,hk)=\beta(g,h)\beta(g,k),\mbox{ for all } g,h,k\in G.
\end{equation}
A bicharacter $\beta$ satisfying
\begin{equation}\label{e4}
 \beta(g,h)=\beta(h,g)^{-1}\mbox{ for all } g,h\in G.
\end{equation}
 is called \emph{skew-symmetric}, or \emph{alternating}.
 
 \medskip
 
 If $G$ is an \emph{abelian} group, then for any a 2-cocycle $\alpha$, the function $\beta(g,h)=\alpha(g,h)\alpha(h,g)^{-1}$ is a skew-symmetric bicharacter.
 
 \medskip
  
 For further information about the cohomology groups, we refer the reader to \cite{rotman2009introduction},\cite{rotman.advanced} and \cite{Brauer}. 
 
 \medskip
  
In this paper, we deal with the \emph{unital} associative algebras over the field $K$. Let $R$ be such an algebra. We say that $R$ is a \emph{$G$-graded algebra} if there exist subspaces $R_{g} \subseteq R$, for each $g \in G$, such that
\[
R = \bigoplus_{g \in G} R_g \quad \text{and} \quad R_{g} R_{h} \subseteq R_{gh}.
\]
The \textit{support} of the graded algebra $R$ is the set $\operatorname{Supp}(R) = \{g \in G \mid R_{g} \neq \{0\}\}$. Since the identity element $1$ of $R$ is an element of $R_{e}$, we have $e\in \operatorname{Supp}(R)$. An ideal $I$ of $A$ is called \emph{graded} if $I=\bigoplus_{g\in G} I\cap R_{g}$. A $G$-graded algebra $R$ is called \emph{graded simple} if $R^{2}\neq \{0\}$ and $R$ has no proper nonzero graded ideals.

Given a $2$-cocycle $\alpha$ on a group $G$ and a $G$-graded algebra $R$, one can define an \emph{$\alpha$-twist} $R^\alpha$ as a vector space $R$ with the operation $g\circ_\alpha h=\alpha(g,h)gh$ for all $a\in R_g$ and $b\in R_h$. An important particular case is where $R=KG$, the group algebra of $G$ with the natural $G$-grading $(KG)_g=Kg$, for any $g\in G$. If $\alpha\in Z^2(G,K^*)$ then the $\alpha$-twist of $KG$ is called an ($\alpha$-)twisted group algebra, denoted by $K^\alpha G$. 
\begin{Remark}
\label{isomorphism.twisted}
    By \cite[Theorem 2.13]{Kochetov.book}, given $\alpha, \alpha'\in Z^{2}(G,K^{\ast})$,  $K^{\alpha}G$ is isomorphic to $K^{\alpha'}G$ (as $G$-graded algebras) if and only if $[\alpha]=[\alpha']$.
\end{Remark}
  
Given a countable set of variables $X=\{x_{1},x_{2},\ldots\}$ let $K\langle X\rangle$ denote the free associative algebra freely generated by $X$. Given an algebra $R$ and $0\neq f(x_{1},\ldots, x_{n})\in K\langle X\rangle$ we say $f$ is a \emph{polynomial identity} for $R$ if $f(r_{1},\ldots, r_{n})=0$, for all $r_{1}$,\dots, $r_{n}\in R$. If $R$ satisfies some polynomial identity we say $R$ is a PI-algebra.  Given PI-algebra $R$, we can form the set $T(R)\subseteq K\langle X\rangle$ of all polynomial identities satisfied by $R$. It can be shown that $T(R)$ is an ideal of $K\langle X\rangle$ and for any endomorphism $\phi\colon K\langle X\rangle\rightarrow K\langle X\rangle$, $\phi(T(R))\subseteq T(R)$. For instance, if $\operatorname{char}(K)=0$ and $E$ is the infinite dimensional Grassmann algebra generated by $\{e_{1},\ldots,e_{n},\ldots\}$, then $T(E)=\langle[x_{1},x_{2},x_{3}]\rangle^{T}$. Another classical example, which requires only that $K$ is infinite, is $UT_{n}(K)$, the algebra of $n \times n$ upper triangular matrices; in this case, it is well-known that $T(UT_{n}) = \langle [x_{1},x_{2}] \cdots [x_{2n-1},x_{2n}] \rangle^{T}$.

Given $n\in \mathbb{N}$, we denote  by $P_{n}$ the vector space of $K\langle X\rangle$ generated by all the monomials $x_{\sigma(1)}\cdots x_{\sigma(n)}$, with $\sigma\in S_{n}$. We define the $n$-th codimension of $R$ to be $c_{n}(R)=\dim P_{n}(R)$, where $P_{n}(R)=P_{n}/P_{n}\cap T(R)$. The PI-exponent of $R$ is defined as
\[
\exp(R)=\lim_{n\rightarrow \infty }\sqrt[n]{c_{n}(R)}.
\]
The existence and integrality of the PI-exponent for the associative algebras over fields of characteristic $0$ has been established by Giambruno and Zaicev in \cite[Theorem 6.5.2]{GZbook}. For instance if $R$ is a central simple algebra, i.e $R$ is simple and $Z(R)\cong K$, then $\exp(R)=\dim (R)$, in particular $\exp(M_{n}(K))=n^{2}$. In \cite{GZbook}, one can find the following calculation of the PI-exponent of a finite dimensional algebra $R$ over $K$. Consider the Wedderburn–Malcev decomposition of $R$, 
$R=B_{1}\oplus \cdots \oplus B_{m}\oplus J(R)$,
where $B_{1},\cdots, B_{m}$ are simple subalgebras of $R$ and $J(R)$ is its Jacobson radical.
Then $\exp(R)$ is equal to the maximal value of
$\dim(B_{i_{1}}\oplus \cdots \oplus B_{i_{p}})$,
where $B_{i_{1}}$,\dots, $B_{i_{p}}\in \{B_{1},\ldots, B_{m}\}$ are distinct and satisfy $B_{i_{1}}J(R)\cdots J(R)B_{i_{p}}\neq \{0\}$.

\section{Regular gradings by non-abelian groups}\label{s3}
Throughout this section, $G$ is a finite group with neutral element $e$. Let us remind the definition \cite[Definition 9]{Eli1} of a regular $G$-graded algebra. Given an $n$-tuple $\textbf{g}=(g_{1},\ldots,g_{n})\in G^{n}$ we denote by $S_{n,\textbf{g}}$ the set of all permutations $\sigma$ in $S_{n}$ such that $g_{\sigma(1)}\cdots g_{\sigma(n)}=g_{1}\cdots g_{n}$. 

\begin{Definition} 
\label{Definition.general.reg}
Let $R$ be a $G$-graded algebra.  We say that $R=\bigoplus_{g\in G}R_{g}$ is a \emph{$G$-graded regular algebra} if the following conditions are satisfied:
\begin{enumerate}
    \item[(1)] For any $n\in \mathbb{N}$ and any $n$-tuple $\mathfrak{g}=(g_{1},\ldots, g_{n})\in G^{n}$, there exists $a_{i}\in R_{g_{i}}$, $1\leq i\leq n$, such that $a_{1}\cdots a_{n}\neq 0$. 
    \item[(2)] Given any $n\in \mathbb{N}$ and any $n$-tuple $\mathbf{g}=(g_{1},\ldots, g_{n})\in G^{n}$, for each $\sigma\in S_{n,\mathbf{g}}$, there exists $\theta(\mathbf{g},\sigma)\in K^{\ast}$ such that for all $a_{i}\in R_{g_{i}}$, $1\leq i\leq n$, we have
    \[
    a_{1}\cdots a_{n}=\theta(\mathbf{g},\sigma)a_{\sigma(1)}\cdots a_{\sigma(n)}.
    \]
\end{enumerate}
The function $\theta(\mathbf{g},\sigma)$ is called the \textit{commutation function} of $R$. 
\end{Definition}
\begin{Remark} If $R$ is a $G$-graded regular algebra, then (1) implies that $\operatorname{Supp}(R) = G$.
\end{Remark}

\begin{Remark} 
\label{recovering}
In our text, given $g\in G$ and $h\in C_{G}(g)$, we write $\theta(g,h)$ instead of $\theta((g,h),(1,2))$. If $G$ is a finite abelian group, written multiplicatively, then the property (2) is replaced by the following classical property (see \cite{regevz2}): 
\begin{enumerate} 
    \item[(2')] for every $g, h\in G$ and every $a\in  R_g,\, b\in R_h$ we have
    \begin{equation}
    \label{eq.2'}
   ab = \theta(g,h)ba
	\end{equation}
\end{enumerate}

 If $R$ is a $G$-graded regular algebra, we denote by $M_{G}^{R}=(\theta(g,h))_{(g,h)}$ the \textit{  decomposition matrix} of $R$. 
    
\end{Remark}

\begin{Remark} Let $G$ be an abelian group, $R$ a $G$-graded algebra, and $\theta\colon G\times G\rightarrow K^{\ast}$ a bicharacter of $G$. If $R$ satisfies only (\ref{eq.2'}), then $R$ is said to be $\theta$-commutative.
	
\end{Remark}

\begin{Example} Given a cocycle $\alpha\in Z^{2}(G,K^{\ast})$, the twisted group algebra $C=K^{\alpha}G$ is spanned over $K$ by a set of invertible elements $\{X_{g}\mid g\in G\}$ subject to the following product: $X_{g}X_{h}=\alpha(g,h)X_{gh}$. It is a $G$-graded regular algebra with the commutation function  $\theta(g,h)=\alpha(g,h)\alpha(h,g)^{-1}$,  for all $h$, $g\in G$.  As we noted earlier, $\theta(g,h)$ is a skew-symmetric bicharacter. For further details, see \cite[Remark 1.1.1]{Eli1}. 
\end{Example}

\begin{Definition} Let $R$ be a $G$-graded regular algebra with the commutation function $\theta$. We say that the regular decomposition of $R$ is \textit{minimal} if $\theta$ is non-degenerate, i.e for any $g\in G$, there exists $h\in C_{G}(g)$ such that $\theta(g,h)\neq 1$, or equivalently (by \cite[Lemma 22]{Eli1}) $\theta(h,g)\neq 1$.
\end{Definition}
\begin{Remark}
   If $G$ is a finite abelian group with the neutral element $e\in G$, and $R$ is a $G$-graded regular algebra with bicharacter $\beta\colon G\times G\rightarrow K^{\ast}$, then the regular decomposition of $R$ is minimal if and only if $\beta(g,h)=1$ for all $h\in G$ implies that $g=e$. Equivalently, $M_{G}^{R}$ has no two columns (equivalently, rows) that are equal.
\end{Remark}

Let $R$ be a $G$-graded regular grading with commutation function $\theta$. Consider the map $\psi\colon  G\rightarrow \{\pm 1\}\cong \mathbb{Z}_{2}$, $g\mapsto \theta(g,g)$. It can be shown $\psi$ is a homomorphism of groups (see proof of  \cite[Lemma 37 ]{Eli1}), and $H_{\psi}:=Ker(\psi)=\{g\in G\mid \theta(g,g)=1\}$  is a subgroup of index at most 2. Suppose $\psi\neq 0$. Then, $G/H_{\psi}\cong \mathbb{Z}_{2}$. In this way, by coarsening, $R=R_{H_{\psi}}\oplus R_{G\setminus H_{\psi}} $, where $R_{H_{\psi}}=\bigoplus_{h\in H_{\psi}}R_{h}$ and $R_{G \setminus H_{\psi}}=\bigoplus_{g\in G\setminus H_{\psi}}R_{g}$,  $R$ becomes a $\mathbb{Z}_{2}$-graded algebra.

The next result generalizes \cite[Lemma 6]{L.P.K.3}.

\begin{Proposition} 
\label{infinite.dimensional}
Let $R$ be a $G$-graded regular grading with the commutation function $\theta$. If $\theta(g,g)=-1$, for some $g\in G$, then $\dim R=\infty$. 
\end{Proposition}
\begin{proof} We use the notation preceding the Proposition. Suppose $\dim R=n<\infty$. Let $g\in G\setminus H_{\psi}$. By regularity, there exist $a_{1}$,\dots, $a_{n+1}\in R_{g}$ such that $a_{1}\cdots a_{n+1}\neq 0$. Notice that, since $a^{2}=0$ for every $a\in R_{g}$, all the elements $a_{i}$ must be distinct. We claim that the set $\mathcal{O}:=\{a_{i_{1}}\cdots a_{i_{r}}\mid 1\leq i_{1}< \cdots <i_{r}\leq n+1,\quad r\leq n+1\}$ is linearly independent. Denote by $q \in \mathbb{N}$ the minimal integer such that there exist elements $u_{1}, \dots, u_{q} \in \mathcal{O}$ and nonzero scalars $\gamma_{1}, \dots, \gamma_{q} \in K$ satisfying
\[
\gamma_{1}u_{1} + \cdots + \gamma_{q} u_{q} = 0.
\]
Of course we can assume $u_{i}\neq u_{j}$, for all $1\leq i\neq j\leq q$. Therefore, there exist $1\leq i\neq j\leq q$, such that some $a_{k}$ appears in the expression of $u_{i}$ but does not appear in the expression of $u_{j}$. Hence, multiplying the expression $\gamma_{1}u_{1}+\cdots +\gamma_{q}u_{q}=0$ by $a_{k}$, we obtain a new expression with fewer terms, all having nonzero coefficients. This contradicts the minimality of $q$. Hence, $\mathcal{O}$ is a linearly independent subset of $R$ with $2^{n+1}-1$ elements, which is a contradiction because $2^{n+1}-1>n$. We conclude $\dim R=\infty$. 
\end{proof}

Recall that the $e$-center of a $G$-graded algebra $R$ is defined by $Z(R)_{e}=Z(R)\cap R_{e}$. Since in a regular grading $R$ we have $R_{e}\subseteq Z(R)$, it follows that the $e$-center of $R$ is $R_{e}$.  From now on, until the end of this section, assume that $\operatorname{char}(K)=0$. The next result shows that, under suitable conditions on the $e$-center of a finite-dimensional $G$-graded regular algebra, it contains a regular subalgebra graded by a finite abelian group.

\begin{Proposition} Let  $G$ be a finite group and $R$ be a finite dimensional $G$-graded regular algebra with the commutation function $\theta$.  Suppose that the $e$-center of $R$ is $K$. Then there exists a maximal finite abelian subgroup $Q$ of $G$, and $\alpha\in Z^{2}(Q,K^{\ast})$ such that $R$ has a graded subalgebra isomorphic to $K^{\alpha}Q$.  
\end{Proposition}
\begin{proof} Let $Q$ be a maximal abelian subgroup of $G$. By \cite[Lemma 22]{Eli1}, the restriction $\theta\!\mid_{Q\times Q}$ is a bicharacter of $Q$, and by regularity, given $q \in Q$ and the pair $(q, q^{-1}) \in Q^{2}$, there exist elements $a_{q} \in R_{q}$ and $a_{q^{-1}} \in A_{q^{-1}}$ such that $0\neq a_{q} a_{q^{-1}} \in R_{e}$. Since $R_{e}=K$, without loss of generality, we may assume that $a_{q} a_{q^{-1}} = 1$, that is, $a_{q^{-1}} = a_{q}^{-1}$. Denote by $\mathcal{B}$ the $Q$-graded subalgebra of $R$ generated by $\{a_{q} \mid q \in Q\}$. Then $\mathcal{B}_{e} = K$, and clearly $\mathcal{B}$ is a $Q$-graded regular algebra. By \cite[Proposition 16]{L.P.K.2}, there exists $\alpha\in Z^{2}(Q,K^{\ast})$ such that $\mathcal{B}$ is isomorphic to a twisted group algebra $K^{\alpha} Q$.
\end{proof}

\begin{Lemma}
\label{Lemma.twisted}
Let  $R$ be a finite dimensional $G$-graded regular algebra with the commutation function $\theta$. If $R$ is graded simple, then $R$ is isomorphic to a twisted group algebra.
\end{Lemma}
\begin{proof} By regularity, we must have $\operatorname{Supp}(R)=G$. By 
 \cite[Theorem 3]{bahturin2005finite}, there exist a subgroup $Q$ of $G$, a 2-cocycle $\alpha\in H^{2}(Q,K^{\ast})$ and an $r$-tuple $\mathbf=(g_{1},\ldots, g_{r}))$, $r\geq 1$, such that $R$ is graded isomorphic to $S=M_{r}(K^{\alpha}Q)$ with the graded basis $ge_{ij}$, $g\in G$, $1\le,i,j\le r$ such that $\deg ge_{ij}=g_{i}^{-1}(\deg g)g_{j}\}$. But if $r>1$, take $a=h\,e_{12}\in S$ and $b=e_{11}$. We have $\deg a=g_{1}^{-1}hg_{2}$, $\deg b=e$ and $ab=0$ while $ba=a\ne 0$, which is a contradiction. Thus $r=1$ and since $\operatorname{Supp}(A)=G$, we have $Q=G$, hence $A\cong K^{\alpha}G$.  
\end{proof}

\begin{Corollary} 
\label{diag.equally.1}
Let $R$ be a finite dimensional $G$-graded regular graded algebra with the commutation function $\theta$. Then, for all $g\in G$, $\theta(g,g)=1$. 
\end{Corollary}

\begin{Remark} 
\label{Remark.Wedderburn.Malcev}
Let $R$ be a finite dimensional $G$-graded algebra. By \cite[Theorem 4.4]{SusanMontgomery}, the Jacobson radical $J(R)$ is a graded ideal of $R$. Since $K$ is an algebraically closed field of characteristic $0$, by \cite[Corollary 2.8]{cstefan1999wedderburn}, the graded version of the Wedderburn--Malcev theorem for finite groups holds for $R$. In other words, there exists a $G$-graded semisimple subalgebra $B$ of $R$ such that $R = B\oplus J(R)$, where $B = D_{1} \oplus \cdots \oplus D_{r}$ and each $D_{i}$ is a $G$-graded simple subalgebra of $D$ for $1 \leq i \leq r$.
\end{Remark}
\begin{Proposition} 
\label{Proposition.KG}
Let $R$ be a finite dimensional $G$-graded regular algebra with the commutation function $\theta$. Then, there exists  a cocycle $\alpha\in Z^{2}(G,K^{\ast})$ such that $K^{\alpha}G$ is a subalgebra of $R$. 
\end{Proposition}
\begin{proof} By Remark \ref{Remark.Wedderburn.Malcev} we can write $R=D\oplus J(R)$, where $D=D_{1}\oplus\cdots  \oplus D_{r}$ direct sum (as algebras) of $G$-graded simple subalgebras of $R$. By Lemma \ref{Lemma.twisted}, for all $1\leq i\leq r$, there exists a subgroup $Q_{i}$  of $G$ and $\alpha_{i}\in Z^{2}(Q_{i},K^{\ast})$ such that $D_{i}$ is graded isomorphic to $K^{\alpha_{i}}Q_{i}$. Suppose $Q_{i}\neq G$, for all $1\leq i\leq r$ and write $G=\{g_{1},\ldots, g_{t}\}$, $g_{1}=e$. In this case, since $D_{\ell}D_{\ell'} = 0$ whenever $\ell \neq \ell'$, we conclude that a product $a_{1}\cdots a_{t}$, with $a_{i} \in R_{g_{i}}$, is nonzero if and only if there exists $1\leq k\leq r$ such that $a_{k} \in J(R)_{g_{k}}$, and in this case $a_{1}\cdots a_{t} \in J(R)$. However, this contradicts the regularity of $R$, since $J(R)$ is a nilpotent ideal. Therefore, there exists $1\leq j\leq r$ such that $Q_{j}=G$. 
\end{proof}
\begin{Corollary} 
\label{dimension.|G|}
Let $R$ be a finite dimensional $G$-graded regular algebra with the commutation function $\theta$. If $\dim R=|G|$, then $R$ is isomorphic to $K^{\alpha}G$ for some $\alpha\in Z^{2}(G,K^{\ast})$. 
\end{Corollary}
\begin{proof}
    This follows directly from Proposition \ref{Proposition.KG}, because there exists a graded subalgebra of $R$ isomorphic to $K^{\alpha}G$ and $\dim K^{\alpha}G=|G|$.
\end{proof}

\begin{Remark}
\label{ray.classes}
Let $R$ be a $G$-graded regular grading with commutation function $\theta$.
\begin{enumerate}
    \item[(a)] Suppose $\dim R<\infty$. Then, by \cite[Lemma 32]{Eli1} there exists $\alpha\in Z^{2}(G,K^{\ast})$ such that the commutation of $K^{\alpha}G$ is $\theta$. 
    \item[(b)] Consider $H_{\psi}$ as in Proposition \ref{infinite.dimensional} and suppose $G\neq H_{\psi}$. Then, by \cite[Lemma 37]{Eli1}, there exists $\alpha\in Z^{2}(G,K^{\ast})$ such that for $\mathcal{B}:=K^{\alpha}G$, with the $\mathbb{Z}_{2}$-grading $\mathcal{B}=\mathcal{B}_{H_{\psi}}\oplus \mathcal{B}_{G\setminus H_{\psi}}$, the Grassmann envelope $E(\mathcal{B})$ is a $G$-graded regular grading with commutation function $\theta$. Moreover, it can be seen that if the regular decomposition of $E(\mathcal{B})$ is minimal, then $\mathcal{B}=\mathcal{B}_{H_{\psi}}\oplus \mathcal{B}_{G\setminus H_{\psi} }$ is a $\mathbb{Z}_{2}$-graded simple algebra.  
    \item[(c)] Consider $R=K^{\alpha}G$. An element $g\in G$ is said to be $\alpha$-regular element if $\alpha(g,h)=\alpha(h,g)$, for all $h\in C_{G}(g)$. A class of conjugate elements of $G$ consisting of $\alpha$-regular elements is called a \text{ray class} with respect to $\alpha$. By \cite[Lemma 34]{Eli1} any ray class produces unique (up to a multiplication by a constant) central element in $R$. It follows that $R$ is simple if and only if the commutation induced by $\alpha$ is non-degenerate. 
\end{enumerate}
\end{Remark}

\begin{Example}
\label{Examples.Schur.Multiplier}
 Given a finite group $G$, we denote by $\mathscr{M}(G)$ its Schur Multiplier (see \cite{rotman2009introduction}).
If $G$ is a cyclic group, with $|G|\geq 3$, then by \cite[Proposition 1.5.5]{Karpilovsky.vol.2} it follows that $\mathscr{M}(G)$ is trivial. Therefore, in light of Remark \ref{isomorphism.twisted} and Remark \ref{ray.classes}, if $R$ is finite dimensional $G$-graded regular algebra with $G$ being a cyclic group of order greater than or equal to $3$, $R$ must be a commutative algebra. If $G\cong S_{3}$, the permutation group on $3$ letters, or $G \cong \mathbb{Q}_{8}$, the quaternion group of order $8$, then by \cite[Corollary 10.1.27]{Karpilovsky.vol.2} and \cite[Proposition 12.5.5]{Karpilovsky.vol.2}, respectively, it also follows that $\mathscr{M}(G)$ is trivial. Thus in these cases, $K^{\alpha}G\cong KG$, for all $\alpha\in Z^{2}(G,K^{\ast})$. 
\end{Example}

\section{Regular quantum commutative algebras}\label{s4} 

In this section, we will work with the we do not assume that our algebras are group-graded. This definition is exactly the original one, given by Regev and Seeman in \cite{regevz2}, although in that paper the authors worked with algebras whose regularity had arisen from group gradings.

First, we recall the notion of a set grading on an algebra.
\begin{Definition}\label{set_grading}
Given an algebra $R$ and a set $S$, we say $R$ is $S$-graded (or  graded by $S$) if there exist vector subspaces $R_{s}\neq \{0\}$, $s\in S$, such that 
\begin{equation}\label{e_set_grading}
   \Gamma: R=\bigoplus_{s\in S} R_{s}
\end{equation}
and for any $s,t\in S$ either $R_{s}R_{t}=\{0\}$ or there exists $f(s,t)\in S$ such that $R_{s}R_{t}\subseteq R_{f(s,t)}$.  A set grading \ref{e_set_grading} is realized as a group-grading by a group $G$ if: (1) $S\subset G$; (2) $R_{g}=R_{s}$, if $g\in S$, and $R_{g}=\{0\}$, otherwise. (3)  If $f(g,h)$ is well-defined then $f(g,h)=gh$.
\end{Definition}

\begin{Definition}\cite[Definition 2.3]{regevz2}\label{dQCD} Let $R$ be an associative algebra. Consider a vector space decomposition 
\begin{equation}\label{eQCD}
 \mathscr{R}\colon\quad R=R_{1}\oplus \cdots \oplus R_{m}.
\end{equation}
Set $S=\{1,\ldots,m\}$. We call (\ref{eQCD}) a \textit{regular quantum commutative decomposition} of $R$ if the following hold:
\begin{enumerate}
    \item[(a)] For any $n\in \mathbb{N}$, given an $n$-tuple $(i_{1},\ldots, i_{n})\in S^{n}$, there exists $a_{i_{j}}\in R_{i_{j}}$, $1\leq j\leq n$, such that \[a_{i_{1}}\cdots a_{i_{n}}\neq 0;\] 
    \item[(b)] There is a function $\theta:S\times S\to K^\ast$ such that for each $i,j$,  $1\leq i,j\leq m$, $a\in R_{i}$ and $b\in R_{j}$,  one has 
        \[
        ab=\theta(i,j)ba.
        \]  
\end{enumerate}
The function $\theta$ is called the \textit{regular quantum commutative function of $R$ with respect to the decomposition $\mathscr{R}$}. The number $m$ is called  the \textit{quantum length} of $\mathscr{R}$. An algebra $R$ admitting $\rR$ and satisfying (b) is called \emph{$(\rR,\theta)$-commutative}. If the \emph{datum} $(\rR,\theta)$ is fixed, we simply say that $R$ is a \emph{regular quantum commutative algebra}. We write $R=(\rR,\theta)$.
\end{Definition}
We denote by $M^{R}=(\theta(i,j))$ the \textit{quantum decomposition matrix} (or just the \textit{decomposition matrix}) of a regular quantum commutative algebra $R=(\mathscr{R},\theta)$. 

\begin{Example}
\label{regular.grading.Regev.Seeman}
If $G$ is a finite abelian group then any $G$-graded regular algebra in the sense of Remark \ref{recovering} is a regular quantum commutative algebra. In this case, the regular quantum commutative function is the skew-symmetric bicharacter on $G$. 
\end{Example}
\begin{Remark}
\label{commutative.}
     It is worth noting that every  regular quantum commutative algebra which is commutative can be viewed, in a trivial way, as a regular graded commutative: if $R=(\mathscr{R},1)$ is a commutative regular quantum commutative algebra, then $R$ becomes regularly graded by the trivial group $\{e\}$. Moreover, it is clear that if $R=(\mathscr{R},\theta)$ is a non-commutative regular quantum commutative algebra, then it cannot be graded by the trivial group.
\end{Remark}


\begin{Lemma}
\label{regular quantum commutative.relations}
Let $R=(\rR,\theta)$ be a regular quantum commutative algebra.  Then, for any $1\leq i,j\leq m$, the \textit{regular quantum commutative relations hold}: $\theta(i,i)^2=1$, $\theta(i,j)=(\theta(j,i))^{-1}$ .
\end{Lemma}
\begin{proof}
    Let $1\leq i,j\leq m$. By regularity, there exist $a_{i}$, $a_{i}'$, $b_{i}\in R_{i}$, $b_{j}\in R_{j}$, such that $a_{i}a_{i}'\neq 0$ and $b_{i}b_{j}\neq 0$. Notice the following
\begin{align*}
    b_{j}b_{i}=\theta(j,i)b_{i}b_{j}=\theta(j,i)\theta(i,j)b_{j}b_{i}
\end{align*}
thus, $\theta(j,i)\theta(i,j)=1$, i.e $\theta(i,j)=(\theta(j,i))^{-1}$. Moreover $a_{i}a_{i}'=\theta(i,i)a_{i}'a_{i}=(\theta_{i,i})^{2}a_{i}a_{i}'$, i.e $(\theta(i,i))^{2}=1$. 
\end{proof}
\begin{Lemma} Let $R=(\rR,\theta)$ be a regular quantum commutative algebra. If there exists $1\leq i\leq m$ such that $\theta(i,i)=-1$, then $\dim R=\infty$. In particular, if $R$ is finite dimensional then $\theta(i,i)=1$, for all $1\leq i\leq m$. 
\end{Lemma}
\begin{proof}
    The proof is exactly the same as in Proposition \ref{infinite.dimensional}. 
\end{proof}
\begin{Lemma}
\label{non.nilp}
Let $R=(\rR,\theta)$ be a finite dimensional regular quantum commutative algebra of quantum length $m$. Then, for each $1\leq i \leq m$ there exist  a non-nilpotent element $w_{i}\in R_{i}$, such that $w_{1}\cdots w_{m}$ is not nilpotent.
\end{Lemma}
\begin{proof} Let $W$ be a (finite-dimensional) subalgebra of $R$ generated by the set $U$ of all products  $t_{1}\cdots t_{m}$ where $t_{i}\in R_{i}$,  $i=1,\ldots m$. A finite basis $\mathcal{W}$ for $W$ can be chosen among the products of the elements of $U$. If all the elements in $\mathcal{W}$ are nilpotent, it follows by \cite[Theorem 2.3.1]{herstein}, that $W$ is a nilpotent subalgebra of $R$. Therefore, there exists $k\in \mathbb{N}$ such that $W^{k}=\{0\}$. Since $R$ is regular, for any natural $n$ there exist $a_i^{(j)}\in R_i$, such that if we set $b_j=a_1^{(j)}\cdots a_m^{(j)}$, $j=1,\ldots n$,  then $b_1\cdots b_n\ne 0$. But each $b_j$ is in $W$ and so $b_1\cdots b_n=0$ for any $n>k$. This contradiction shows that at least one of the elements of $\mathcal{W}$ is not nilpotent.

Suppose that $x$ is a non-nilpotent element of $\mathcal{W}$. We can write $x = w u$, where $w = w_1 \dots w_m$ with $w_i \in R_i$ and $u\in\mathcal{W}$. Due to the quantum commutativity, we have
$x^k = (w u)(w u) = \alpha w^k u^k$, for some $\alpha \in K^{\ast}$. Thus, if $w^k = 0$ for some $k$, then also $x^k = \lambda w^k u^k = 0$, which is not the case. So $w=w_1\cdots w_m$ is not nilpotent. Now for each $\ell$ there exists $\mu\in K^\ast$ such that $w^\ell=\mu w_1^\ell\cdots w_m^\ell$. It follows that if one of $w_i$ is nilpotent then also $w$ is nilpotent, which was proven to be false. So we have found non-nilpotent $w_i\in R_i$, $i=1,\ldots, m$, such that their product $w_1\cdots w_m$ is not nilpotent. The proof is complete.
\end{proof}

\begin{Proposition}
	\label{commutattive.PI}
	Let $R=(\rR,\theta)$ be a regular quantum commutative algebra of quantum length $m$. Then, $R$ is a PI-algebra. 
\end{Proposition}
\begin{proof}  Let $n\in \mathbb{N}$ and 
	\[
	g(x_{1},\ldots, x_{n})=\sum_{\sigma\in S_{n}} \lambda_{\sigma}x_{\sigma(1)}\cdots x_{\sigma(n)},\quad \lambda_{\sigma}\in K,
	\]
	a multilinear polynomial in $K\langle X\rangle$. Let $\ell_{1},\ldots, \ell_{n}\in \{1,\ldots, m\}$ and let $b_{1}\in R_{\ell_{1}},\dots, b_{n}\in R_{\ell_{n}}$ be homogeneous elements. Then, for any $\sigma\in S_{n}$, by regularity we have
	\[
	b_{\sigma(1)}\cdots b_{\sigma(n)}=\Lambda_{\sigma}(\ell_{1},\ldots, \ell_{n})(b_{1}\cdots b_{n}),
	\]
	where 
	\[
	\Lambda_{\sigma}(\ell_{1},\ldots, \ell_{n})=\prod_{\substack{i<j \\ \sigma(i)>\sigma(j)}}\theta(\ell_{\sigma(i)},\ell_{\sigma(j)}).
	\]
	Therefore,
	\begin{equation}
		\label{eq.multilinear}
		g(b_{1},\ldots, b_{n})=\Bigg(\sum_{\sigma\in S_{n}}\Lambda_{\sigma}(\ell_{1},\ldots, \ell_{n})\lambda_{\sigma}\Bigg)(b_{1}\cdots b_{n}).    
	\end{equation}
	
	Choose $n$ large enough such that $n!>m^{n}$, and consider the linear homogeneous system 
	\begin{equation}
		\label{system}
		\sum_{\sigma\in S_{n}}\Lambda_{\sigma}(\ell_{1},\ldots, \ell_{n})c_{\sigma}=0,\quad \ell_{1},\ldots, \ell_{n}\in \{1,\ldots, m\},
	\end{equation}
	where $c_{\sigma}$ are the unknowns. Since (\ref{system}) has $m^{n}$ equations and $n!>m^{n}$, it admits a nonzero solution $(\lambda_{\sigma})_{\sigma\in S_{n}}$, and thus by (\ref{eq.multilinear}) $0\neq g(x_{1},\ldots, x_{n})\in T(R)$.
\end{proof}

\begin{Example} The Weyl algebra $\mathcal{A}_{1} = K\langle x,y\rangle/(xy - yx - 1)$ does not admit a regular quantum commutative decomposition. The same holds for the universal enveloping algebra $U(\mathfrak{g})$ of any non-abelian Lie algebra $\mathfrak{g}$. 
\end{Example}

\begin{Theorem}
\label{nth.root}
Let $R=(\rR,\theta)$ be a regular quantum commutative algebra of quantum length $m$ and suppose $\operatorname{char}(K)=0$. Then all $\theta(i,j)$ are roots of unity.
\end{Theorem}
\begin{proof} Let $1\leq i,j\leq m$ and denote by $\mathfrak{V}:=\operatorname{var}(R)$ the variety generated by $R$. First, suppose that $UT_{2}(K)\in \mathfrak{V}$. Since $UT_{2}(K)$ satisfies $[x_{1},x_{2}][x_{3},x_{4}]=0$, then also $R$ satisfies $[x_{1},x_{2}][x_{3},x_{4}]= 0$. By regularity, there exist $a\in R_{i}$ and $b\in R_{j}$ such that $ba\neq 0$. In this case,
\[
0=[a,b][a,b]=(ab-ba)(ab-ba)=(\theta(i,j)-1)^{2}ba,
\]
hence for all $i,j$, $\theta(i,j)=1$, as claimed. 

In the remaining case, $UT_{2}(K)\notin \mathfrak{V}$. By \cite[Theorem 7.3.1]{GZbook}, $R$ satisfies a polynomial identity of the form
\[
f(x,y)=\sum_{s=1}^{n+1}\alpha_{s}y^{s}xy^{((n+1)-s)}= 0,
\] are nonzero. The complete linearization of $f$
\[
\widehat{f}(x,y_{1},\ldots, y_{n+1})=\sum_{s=1}^{n+1}\sum_{\sigma\in S_{n+1}}\alpha_{s}y_{\sigma(1)}\cdots y_{\sigma(s)}xy_{\sigma(s+1)}\cdots y_{\sigma(n+1)}
\]
 is a polynomial identity of $R$. By regularity, there exist $a\in R_{i}$, $b_{1,1}$, $b_{1,2}$, $b_{2,1}$, $b_{2,2}$,\dots, $b_{n+1,1}$, $b_{n+1,2}\in R_{j}$ such that 
\[
ac_{1}\cdots c_{n+1}\neq 0,\quad \text{where} \quad c_{k}:=b_{k,1}b_{k,2},\quad 1\leq k\leq n+1.
\]
Since $\theta(j,j)^{2}=1$, given any $1\leq k,l\leq n+1$, we obtain
\begin{align*}
c_{k}c_{l} &=(b_{k,1}b_{k,2})(b_{l,1}b_{l,2})\\
&=\theta(j,j)^{4}(b_{l,1}b_{l,2})(b_{k,1}b_{k,2})=(b_{l,1}b_{l,2})(b_{k,1}b_{k,2})=c_{l}c_{k}.
\end{align*}
 Thus, the elements $c_{1}$,\dots, $c_{n+1}$ commute, and setting $\xi:=\theta(j,i)$, we get
 \begin{align*}
     0 &=\widehat{f}(a,c_{1},\ldots, c_{n+1})=\sum_{s=1}^{n+1}\sum_{\sigma\in S_{n+1}}\alpha_{s}c_{\sigma(1)}\cdots c_{\sigma(s)}ac_{\sigma(s+1)}\cdots c_{\sigma(n+1)}\\
     &= \sum_{s=1}^{n+1} \alpha_{s}\xi^{2s}a\Big(\sum_{\sigma\in S_{n+1}}c_{\sigma(1)}\cdots c_{\sigma(n+1)}\Big)=\Big(\sum_{s=1}^{n+1}\alpha_{s}\xi^{2s}(n+1)!\Big)ac_{1}\cdots c_{n+1} 
 \end{align*}
 since $ac_{1}\cdots c_{n+1}\neq 0$, it follows that $\sum_{s=1}^{n+1}\alpha_{s}\xi^{2s}=0$, i.e $\xi^{2}\Big(\sum_{s=1}^{n+1}\alpha_{s}\xi^{2s-2}\Big)=0$, and since $\xi\neq 0$ we have $\sum_{s=1}^{n+1}\alpha_{s}\xi^{2s-2}=0$.

 We proceed in the following way. Given $t\in \{4,6,\ldots, 2(n+1)\}$, by regularity there exists $a_{t}\in R_{i}$, $b_{1,1}$,\dots, $b_{1,t}$,\dots, $b_{n+1,1}$,\dots, $b_{n+1,t}\in R_{j}$ such that $a_{t}d_{1}\cdots d_{n+1}\neq 0$, where
\[
d_{k}:=b_{k,1}\cdots b_{k,t},\quad \text{for all}\quad 1\leq k\leq n+1
\]
and since $t$ is even, it can be seen that $d_{1}$,\dots, $d_{n+1}$ commute. Hence
\begin{align*}
    0 &=\widehat{f}(a_{t},d_{1},\ldots, d_{n+1})=\sum_{s=1}^{n+1}\sum_{\sigma\in S_{n+1}}\alpha_{s}d_{\sigma(1)}\cdots d_{\sigma(s)}a_{t}d_{\sigma(s+1)}\cdots d_{\sigma(n+1)}\\
     &= \sum_{s=1}^{n+1} \alpha_{s}\xi^{ts}a_{t}\Big(\sum_{\sigma\in S_{n+1}}d_{\sigma(1)}\cdots d_{\sigma(n+1)}\Big)=\Big(\sum_{s=1}^{n+1}\alpha_{s}\xi^{ts}(n+1)!\Big)a_{t}d_{1}\cdots d_{n+1} 
\end{align*}
which implies that $\sum_{s=1}^{n+1}\alpha_{s}\xi^{ts}=0$. Since $\xi\neq 0$, we get $\sum_{s=1}^{n+1}\alpha_{s}\xi^{ts-t}=0$. Now $(\alpha_1,\ldots,\alpha_{n+1})$ is a nonzero solution for the following system of linear equations:
\begin{eqnarray*}
  \alpha_{1}+\alpha_{2}\xi^{2}+\cdots+\alpha_{n+1}\xi^{2n} =  & 0 \\
  \alpha_{1}+\alpha_{2}\xi^{4}+\cdots+\alpha_{n+1}\xi^{4n}  = & 0 \\
     \ldots\ldots\ldots\ldots\ldots\ldots\ldots\ldots\ldots& \\
  \alpha_{1}+\alpha_{2}\xi^{2(n+1)}+\cdots+\alpha_{n+1}\xi^{2(n+1)n}  =  &0
\end{eqnarray*}

The determinant of this system is the Vandermonde's determinant $V(\xi^2,\xi^4,\ldots,\xi^{2(n+1)})$. Since the system has nonzero solution, the determinant is zero, hence for some $q\le k<\ell\le n+1$, we must have $\xi^{2k}=\xi^{2\ell}$. It follows that $\xi$ is a root of 1.
\end{proof}

\begin{Remark} 
In what follows, it is worth recalling that every finite-dimensional unital algebra can be regarded as a subalgebra of an algebra $\mathrm{End}\,(V)$ of endomorphisms of a finite-dimensional vector space $V$. One can also view $R$ as a subalgebra of the matrix algebra $M_n(K)$ for  $n=\dim V$. In this case, we will treat each element $u\in R$ as an $n\times n$ matrix.
\end{Remark}

Assuming the finite dimensional case, we can provide an alternative proof of the above result, which works over any field.

\begin{Theorem} 
\label{general.finite.dimensional.case}
Let $V$ be a vector space over $K$, $\dim V=n<\infty$. Let $f$, $g\in\operatorname{End}(V)$ be such that all $f^{n},g^{n},(fg)^{n}$ are nonzero. If $\xi \in K^{\ast}$ is such that $gf=\xi fg$, then $\xi^{n}=1$.
\end{Theorem}
\begin{proof} If $\overline{K}$ is an algebraic closure of $K$ then the natural extensions $\overline{f}=\mathrm{id}_{\overline{K}}\otimes_K f$ and $\overline{g}=\mathrm{id}_{\overline{K}}\otimes_K g$ act on $\overline{V}=\overline{K}\otimes_K V$, satisfying the same conditions as those assumed for $f$ and $g$. Thus, in what follows, we assume that $K$ is algebraically closed.

     Given a characteristic root $\lambda$ of $g$, the root space $V_{\lambda}$ of $g$ is moved into the root space $V_{\xi\lambda}$. Indeed, a vector $v\in V$ is in $V_{\lambda}$ iff $(g-\lambda\operatorname{id}_V)^{n}(v)=\{0\}$. Now 
  \[
  f(g-\lambda\operatorname{id}_V)=\xi^{-1} gf-\lambda f=\xi^{-1}(g-(\xi\lambda)\operatorname{id}_V)f
  \]
  and so 
\[
  f(g-\lambda\operatorname{id}_V)^{n}(v)=\xi^{-n}(g-(\xi\lambda)\operatorname{id}_V)^{n}f(v)
  \]
  proving that $f(v)\in V_{\xi\lambda}$. 

For any root $\lambda$ of $g$ we consider $V(\lambda)$ which is the sum of all $V_{\xi^{k}\lambda}$ where $k=0,1,2,\ldots$. Obviously, $V$ is the sum of all $V(\lambda)$. By our calculation above, each $V(\lambda)$ is invariant under $f$. Since $f$ is not nilpotent, there is an eigenvalue $\lambda$ such that $V(\lambda)$ contains an eigenvector $u$ for $f$ with nonzero eigenvalue $\mu$: $f(u)=\mu u$.

Suppose $u\in V(\lambda)$ where $\lambda\ne 0$. There are linearly independent $u_{1},\ldots,u_{k}$, $u_{i}\in V_{\xi^{\ell_i}\lambda}$ such that 
  \begin{equation}\label{1}
   u=u_{1}+u_{2}+\cdots+u_{k}\quad \mbox{ where }\quad 0\neq u_{i}\in V_{\xi^{\ell_i}\lambda},\:1\leq \ell_1<\ell_2<\cdots\ell_k\leq n.
  \end{equation}
  Recalling $gf=\xi fg$, we obtain
 \begin{equation}\label{2}
  fg(u)=\lambda f(\xi^{\ell_{1}}u_{1}+\cdots+\xi^{\ell_{k}}u_{k})=(\lambda\xi^{\ell_{1}})f(u_{1})+\cdots+(\lambda\xi^{\ell_{k}})f(u_{k})
  \end{equation}
  and
  \begin{equation}\label{3}
  gf(u)=\mu g(u_{1}+\cdots+u_{k})=\mu(\lambda\xi^{\ell_{1}})u_{1}+\cdots+\mu(\lambda\xi^{\ell_{k}})u_{k}
  \end{equation}
  Now we have
   \[
   f(u_{1})\in V_{\xi^{\ell_{i_{1}+1}}\lambda},\ldots, f(u_{k})\in V_{\xi^{\ell_{i_{k}+1}}\lambda}.
   \]
   Since $\mu(\lambda\xi^{\ell_1})u_{1}\in V_{\xi^{\ell_{1}}\lambda}$ is the only nonzero root vector with root $\lambda\xi^{\ell_1}$ in (\ref{2}), this root must be equal to one of the roots
    \[
   \xi^{\ell_{i_{1}+1}}\lambda,\ldots,\xi^{\ell_{i_{k}+1}}\lambda
   \]
  appearing in (\ref{3}). Then $\lambda\xi^{\ell_{1}}=\lambda\xi^{\ell_{i}+1}$ where $i\in \{1,\ldots, k\}$. Then $\xi^{\ell_{1}}=\xi^{\xi_{i}+1}$ and so $\xi$ is a root of $1$, and by Cayley--Hamilton Theorem we get $\xi^{n}=1$.

   Finally, suppose $f$ is nilpotent on $W=\sum_{\lambda\neq 0}V(\lambda)$. Then $g^{n}(V)\subset W$ and $f^{n}(W)=\{ 0\}$. As a result, $g^{n}f^{n}=0$ hence $(fg)^{n}=0$, a contradiction. The proof is complete.
\end{proof}
\begin{Corollary} Let $R=(\rR,\theta)$ be a finite dimensional regular quantum commutative algebra of quantum length $m$. Then, all $\theta(i,j)$ must be roots of 1.
\end{Corollary}

 \begin{Corollary} Let $R=M_{n}(K)$ and $\mathscr{R}\colon R=R_{1}\oplus \cdots \oplus R_{m}$ be a regular quantum commutative decomposition of quantum length $m$ with regular quantum function $\theta(i,j)$. Then, all $\theta(i,j)$ are $n$-th roots of 1.
\end{Corollary}

\begin{Corollary} 
\label{all.root}
Let $R=(\rR,\theta)$ be a finite dimensional regular quantum commutative algebra of quantum length $m$.  Given $1\leq i\leq m$, there exists $k_{i}\in \mathbb{N}$ such that $\theta(i,j)^{k_{i}}=1$, for all $1\leq j\leq m$.
\end{Corollary}
\begin{proof}
    Given $1\leq i\leq m$, we know that $\theta(i,1)$,\dots, $\theta(i,m)$ are roots of unit. If $\theta(i,j)$ is a $q_{j}$-th root of unit, for $1\leq j\leq m$, we take $k_{i}:= q_{1}\cdots q_{m}$, then $\theta(i,j)^{k_{i}}=1$, for all $1\leq j\leq m$.
\end{proof}

\section{Bahturin-Regev Conjecture}\label{s5}

\begin{Proposition} 
\label{central.invertibility}
Let $R=(\rR,\theta)$ be a finite dimensional regular quantum commutative algebra of quantum length $m$. Suppose $Z(R)= K$ and $w_{i}\in R_{i}$, $1\leq i\leq m$, are such that $w_{1}\cdots w_{m}$ is non-nilpotent. Then, for any $1\leq i\leq m$, $w_{i}$ is invertible and there exists $k_{i}\in \mathbb{N}$ such that $w_{i}^{k_{i}}=\mu_{i}1_{R}$, for some $\mu_{i}\in K^{\ast}$. 
\end{Proposition}
\begin{proof} Take $1\leq \ell\leq m$,  by Corollary \ref{all.root}, there exists $k_{\ell}\in \mathbb{N}$ such that $\theta(\ell,j)^{k_{\ell}}=1$, for all $1\leq j\leq m$. Then, given $1\leq j\leq m$ and $u_{j}\in R_{j}$  we have
\[
w_{\ell}^{k_{\ell}}u_{j}=u_{j}w_{\ell}^{k_{\ell}}.
\]
Thus, if $u=u_{1}+\cdots +u_{m}\in R$, with $u_{i}\in R_{i}$, we get
\[
w_{\ell}^{k_{\ell}}u=\sum_{i=1}^{m}w_{\ell}^{k_{\ell}}u_{i}=\sum_{i=1}^{m}u_{i}w_{\ell}^{k_{\ell}}=uw_{\ell}^{k_{\ell}}.
\]
 Since $w_{1}\cdots w_{m}$ is non-nilpotent, we conclude $0\neq w_{\ell}^{k_{\ell}}\in Z(R)= K$, i.e., $w_{\ell}^{k_{\ell}}=\mu_{\ell}1_{R}$, $\mu_{\ell}\in K^{\ast}$. 
\end{proof}

 \begin{Remark}
	\label{dimension}
	Let $R=(\rR,\theta)$ be a finite dimensional regular quantum commutative algebra of quantum length $m$. Suppose $Z(R)= K$ and $w_{i}\in R_{i}$, $1\leq i\leq m$, are such that $w_{1}\cdots w_{m}$ is non-nilpotent. In view of Proposition \ref{central.invertibility}, we have an alternative way to prove that all $\theta(i,j)$ must be $\dim(R)$-th roots of unity. Indeed, given $1\leq i,j\leq m$ we get
	\begin{align*}
	  \det(w_{i})\det(w_{j}) &= \det(\theta(i,j)w_{j}w_{i})\\
		&= \Big(\theta(i,j)^{\dim(R)}\Big)\det(w_{j})\det(w_{i})
	\end{align*}
	since $w_{i}$ and $w_{j}$ are invertible, we conclude $\theta(i,j)^{\dim(R)}=1$.
\end{Remark}

\begin{Corollary} Let $R=(\rR,\theta)$ be a finite dimensional regular quantum commutative algebra of quantum length $m$. Suppose $Z(R)= K$ and $w_{i}\in R_{i}$, $1\leq i\leq m$, are such that $w_{1}\cdots w_{m}$ is non-nilpotent. If either $\operatorname{char}(K)=0$ or $\operatorname{char}(K)\nmid \dim(R)$, then $w_{1},\dots, w_{m}$ are diagonalizable.
\end{Corollary}

\begin{proof}  
	By Remark \ref{dimension}, it follows that for any $1\leq i\leq m$ there exists $\lambda_{i}\in K^{\ast}$ such that $w_{i}^{\dim(R)}=\lambda_{i} 1_{R}$. Then, given $1\leq i\leq m$, if $m_{w_{i}}(x)$ is the minimal polynomial of $w_{i}$, we conclude
	\[
	m_{w_{i}}(x)\mid (x^{\dim (R)}-\lambda_{i}).
	\]	
	Let $h(x):=x^{\dim (R)}-\lambda_{i}$ and $h'(x)=\dim(R)x^{\dim(R)-1}$ be its formal derivative. Since either $\operatorname{char}(K)=0$ or $\operatorname{char}(K)\nmid \dim(R)$, it follows that $h'(x)=0$ if and only if $x=0$. Therefore, $h(x)$, and consequently $m_{w_{i}}(x)$, has no multiple roots, which implies that $w_{i}$ is diagonalizable.
\end{proof}

Let $R=(\rR,\theta)$ be a regular quantum commutative algebra. We say that the regular quantum commutative decomposition of $R$ is \textit{ minimal} if no two columns (equivalently, two rows) with different numbers in the quantum decomposition matrix $M^{R}$ are equal. If the regular quantum decomposition of $R$ is not minimal, then $\det M^{R}=0$. So if $\det M^{R}\neq 0$, then the regular quantum decomposition of $R$  is minimal.  In the case where the regular quantum decomposition of $R$ is minimal, we simply say that $\mathscr{R}$ is minimal.

\begin{Conjecture}
\label{Conjecture.Bahturin.Regev}
(Bahturin-Regev,\cite{bahturin2009graded}) Let $R=(\rR,\theta)$ be a regular quantum commutative algebra of quantum length $m$. 
\begin{enumerate}
    \item $\rR$ is minimal if and only if $\det M^{R}\neq 0$.
    \item If $\rR$ is minimal, the quantum length $m$, as well as the determinant of the corresponding matrix, are invariants, namely are independent of the particular minimal decomposition.
\end{enumerate}
    
\end{Conjecture}

In \cite{Eli1}, Aljadeff and David gave a positive answer to Conjecture \ref{Conjecture.Bahturin.Regev} in the case of regular gradings when $\operatorname{char}(K)=0$. In \cite{LPK.1}, it was shown that the conjecture fails for regular gradings over any fields whose characteristic is strictly greater than $2$. More precisely, the main result of \cite{LPK.1} states that if $F$ is a field with $\operatorname{char}(F)>2$, then one can construct a finite abelian group $H$ and an $H$-graded regular algebra $R$ over $F$ such that $\det M^{R}=0$, while the regular decomposition of $R$ is minimal. 

However, one may impose certain conditions on the characteristic of the field so that the first part of the conjecture holds for a given abelian group and a regular grading by this group. For the sake of completeness, we present this result below, although the idea is essentially the same as in \cite{Eli1} and \cite{LPK.1}.

\begin{Proposition}
\label{Bahturin.Regev.Grading}
Let $G$ be a finite abelian group, $K$ a field with $\operatorname{char}(K)\nmid |G|$, $\beta\colon G\times G\rightarrow K^{\ast}$ a skew-symmetric bicharacter on $G$. If $R$ is a $G$-graded regular algebra with the commutation  factor $\beta$, then the regular decomposition of $R$ is minimal if and only if $\det M^{R}\neq 0$. 
\end{Proposition}
\begin{proof} It is enough to show that if the regular decomposition of $R$ is minimal, then $\det M^{R}\neq 0$. Assuming this, for any $a\in G$, $a\neq e$, there exists $b\in G$ such that $\beta(a,b)\neq 1$. It follows that
\[
    \sum_{g\in G}\beta(a,g)=\sum_{g\in G}\beta(a,b)\beta(a,gb^{-1})=\beta(a,b)\sum_{g\in G}\beta(a,gb^{-1})
\]
since $\beta(a,b)\neq 1$, it follows that 
\begin{equation}
    \label{aux.k.1}
\sum_{g\in G}\beta(a,g)=0.    
\end{equation}
 If $n_{(g,h)}$ is the $(g,h)$-entry of $(M^{R})^{2}$, then
\begin{align*}
    n_{(g,h)} &=\sum_{c\in G}\beta(g,c)\beta(c,h)=\sum_{c\in G}\beta(g,c)\beta(g^{-1},c)\beta(g^{-1},c)^{-1}\beta(c,h)\\
    &= \sum_{c\in G}\beta(e,c)\beta(c,g^{-1})\beta(c,h)=\sum_{c\in G}\beta(c,g^{-1}h)\\
\end{align*}
by (\ref{aux.k.1}) we get
\[
n_{(g,h)}=\begin{cases}
    |G|,\quad \text{if $g=h$}\\
    0,\quad \text{if $g\neq h$}.
\end{cases}
\]
 Thus, $(M^{R})^{2}=|G|I_{|G|}$, where $I_{|G|}$ is the $|G|\times |G|$-identity matrix. Because $\operatorname{char}(K)\nmid |G|$ we conclude $\det M^{R}\neq 0$. 
\end{proof}

\begin{Lemma}
\label{result.m.dim(R)}
Let $R=(\rR,\theta)$ be a finite dimensional regular quantum commutative algebra of quantum length $m$ (see Definition \ref{dQCD}) such that $\mathscr{R}$ is minimal, $Z(R)= K$ and $m=\dim (R)$. Set $S=\{1\ldots,m\}$. Then, there exists a map $\star\colon S\times S\rightarrow S$ making  $S$ a finite abelian group and $\mathscr{R}$ a $G$-grading. Moreover,  given any $1\leq j,j,k\leq m$, it is true that 
\[
\theta(i,j)\theta(i,k)=\theta(i,j\star k).
\]
\end{Lemma}
\begin{proof}

By Lemma \ref{non.nilp}, there exists $w_{1}\in R_{1}$,\dots, $w_{m}\in R_{m}$ such that $w_{1}\cdots w_{m}$ is not nilpotent.  Since $m=\dim R$,  $w_{1},\dots, w_{m}$ is a basis of $R$ and $R_{i}=\operatorname{Span}_{K}\{w_{i}\}$, for all $1\leq i\leq m$. 

For each pair $1\leq j,k\leq m$ we choose a minimal $t$, $1\leq t\leq m$,  such that 
\[
w_{j}w_{k}=\gamma_{1}w_{s_{1}}+\cdots +\gamma_{t}w_{s_{t}},\mbox{ where }\gamma_{1},\dots, \gamma_{t}\in K^{\ast},\mbox{ and } s_{\ell}\neq s_{\ell'}\mbox{ for }\ell\neq \ell'.
\]
If $t=1$, then $R_{j}R_{k}\subseteq R_{s_{1}}$ and we set $j\star k:= s_{1}$. Otherwise, given $1\leq i\leq m$, we write
\begin{align*}
w_{i}(w_{j}w_{k}) &=w_{i}(\gamma_{1}w_{s_{1}})+\cdots +w_{i}(\gamma_{t}w_{s_{t}})\\ & = \theta(i,s_{1})(\gamma_{1}w_{s_{1}})w_{i}+\cdots +\theta(i,s_{t})(\gamma_{t}w_{s_{t}})w_{i}.
\end{align*}
On the other hand 
\begin{align*}
    w_{i}(w_{j}w_{k}) &= \theta(i,j)\theta(i,k) (w_{j}w_{k})w_{i}\\
 &= \theta(i,j)\theta(i,k)(\gamma_{1}w_{s_{1}})w_{i}+\cdots +\theta(i,j)\theta(i,k)(\gamma_{t}w_{s_{t}})w_{i}.
 \end{align*}  
By  Proposition \ref{central.invertibility}, $w_{i}$ is invertible, thus we obtain 
\begin{equation}
\label{minimal.equality}
    \theta(i,s_{1})=\cdots =\theta(i,s_{t})=\theta(i,j)\theta(i,k).
\end{equation}
In particular the rows $(\theta(i,s_{1}))_{1\leq i\leq m}$ and $(\theta(i,s_{2}))_{1\leq i\leq m}$ of $M^{R}$ are equal, which is a contradiction with the minimality of $\mathscr{R}$. Since $1\leq j,k\leq m$ are arbitrary, we conclude that there exists a binary operation $\star\colon S\times S\rightarrow S$ such that
\[
R_{j}R_{k}\subseteq R_{j\star k},\mbox{ for all } 1\leq j,k\leq m. 
\]
Consequently, $\mathscr{R}$ is a set grading.   Moreover, by (\ref{minimal.equality}) we have
\[
\theta(i,j)\theta(i,k)=\theta(i,j\star k),\quad \text{for all }\quad 1\leq i\leq m.
\]

So we have shown that $G=(S,\star)$ is a groupoi. It follows from the quantum commutativity that $G$ is commutative. Now, since $R$ is associative, given $1\leq i,j,k\leq m$, we have
\[
w_{i}(w_{j}w_{k})=(w_{i}w_{j})w_{k},
\]
which yields $i\star (j\star k)=(i\star j)\star k$. Thus, $\star$ is associative and $G$ is a semigroup. Finally, by Proposition \ref{central.invertibility}, for any $1\leq i\leq m$ there exists $k_{i}\in \mathbb{N}$ such that $w_{i}^{k_{i}}=\mu_{i}1_{R}$, which guarantees the existence of an inverse for $i$ with respect to $\star$ and shows that $1_{R}$ lies in a homogeneous component, say $s_{0}\in S$, and $R_{s_{0}}=\operatorname{Span}_{K}\{1_{R}\}$. It follows that $s_{0}$ is the neutral element of $G$, $G$ is a finite abelian group, and $\mathscr{R}$ is a $G$-grading.
\end{proof}

\begin{Corollary}
\label{twisted.group}
Let $R=(\rR,\theta)$ be a finite dimensional regular quantum commutative algebra of quantum length $m$ such that $\mathscr{R}$ is minimal, $Z(R)= K$ and $m=\dim (R)$. Then, there exists  a finite abelian group $G$ such that $\mathscr{R}$ is a $G$-regular grading, $M_{G}^{R}=M^{R}$ and $\theta(i,j)$ is a skew-symmetric bicharacter of $G$.   
\end{Corollary}
\begin{proof}

By Lemma \ref{result.m.dim(R)}, there exists a binary operation $\star\colon S\times  S\rightarrow  S$  on $S=\{1,\ldots, m\}$ such that $G:=(S,\star)$ is finite abelian group and $\mathscr{R}$ is a $G$-grading. Define 
\[
\beta\colon G\times G\rightarrow K^{\ast}\quad (i,j)\mapsto \theta(i,j).
\]
By the same lemma, we know that for any $1\leq i,j,k\leq m$
\begin{equation}
\label{one.linearity}
    \theta(i,j)\theta(i,k)=\theta(i,j\star k).
\end{equation}
On the other hand, by the quantum commutation relations in Lemma \ref{regular quantum commutative.relations}, we have the skew-symmetry of $\beta$, and by inverting both sides of (\ref{one.linearity}) we obtain
\[
\theta(j,i)\theta(k,i)=\theta(j\star k,i).
\]
Therefore $\beta$ is a bicharacter of $G$, $\mathscr{R}$ is a $G$-regular grading with bicharacter $\beta$ and $M_{G}^{R}=M^{R}$.
\end{proof}

\subsection{Simple regular quantum commutative algebras}\label{ssSQCA}

\begin{Lemma}
\label{weaky.hipothesis}
Let $R=M_{n}(K)$ and let $R=R_{1}+ \cdots + R_{m}$ be a decomposition of $R$ as a (not necessarily direct) sum of vector subspaces satisfying the following conditions:
\begin{enumerate}
\item $R_{i}\neq \{0\}$ for all $1\leq i\leq m$;
\item for all $1\leq i,j\leq m$, there exists a root of unity $\theta(i,j)\in K^{\ast}$ such that
\[
a_{i}a_{j}=\theta(i,j)a_{j}a_{i}
\]
for all $a_{i}\in R_{i}$ and $a_{j}\in R_{j}$.
\end{enumerate}
Under these conditions, there exist $v_{1}\in R_{1}$,\dots, $v_{m}\in R_{m}$ such that $R_{i}=\operatorname{Span}_{K}\{v_{i}\}$,  and there exists $\lambda_{i}\in K^{\ast}$ such that $v_{i}^{n}=\lambda_{i}I_{n}$. Moreover if $R=R_{1}+\cdots+R_{m}$ is a direct sum, then it is a regular quantum commutative decomposition of $R$ and $m=n^{2}$. 
\end{Lemma}
\begin{proof}   First, suppose there exists $1\leq i\leq m$ such that $R_{i}R_{j}= \{0\}$, and consequently $R_{j}R_{i}=\{0\}$, for all $j\neq i$. In this case, we consider an ideal $\mathcal{I}$ generated by $R_{i}$. Since $R$ is simple and $R_{i}\neq 0$, it follows that $\mathcal{I}=R$. Since $R_{i}R_{j}=0$, for all $j\neq i$, the following is true. For any $u=u_{1}+\cdots +u_{m}$, and $\widetilde{u}=\widetilde{u}_{1}+\cdots +\widetilde{u}_{m}$, with $u_{k}$, $\widetilde{u}_{k}\in R_{k}$, given any $a_{i}\in R_{i}$ we have
\[
ua_{i}\widetilde{u}=\sum_{k,l=1}^nuu_{k}a_{i}\widetilde{u}_{l}=u_{i}a_{i}\widetilde{u}_{i}.
\]
Therefore,\emph{} if $I_{n}$ is the $n\times n$ identity matrix of $R$, we conclude $I_{n}$ is a linear combination of products of elements from $R_{i}$. But then $I_{n}a_{j}=0$, for all $a_{j}\in R_{j}$, $j\neq i$, i.e $R_{j}=\{0\}$ for all $j\neq i$, which is a contradiction. Thus, there exists $j\neq i$ such that $R_{i}R_{j}\neq \{0\}$. 

Therefore, given any $1\leq i\leq m$, there exists $v_{i}\in R_{i}$, $j\neq i$, and $v\in A_{j}$ such that $v_{i}v\neq 0$.  By Theorem \ref{general.finite.dimensional.case}, all $\theta(i,j)$ are $n$-th roots of 1. In this case,
\[
v_{i}^{n}a_{k}=\theta(i,k)^{n}(a_{k}v_{i}^{n})=a_{k}v_{i}^{n},\quad \text{for all}\quad 1\leq k\leq m\quad  \text{and}\quad a_{k}\in A_{k}
\]
and it implies that $v_{i}^{n}=\lambda_{i} I_{n}$ for some $\lambda_{i}\in K^{\ast}$. 

Thus, there exists invertible elements $v_{1}\in R_{1}$,\dots, $v_{m}\in R_{m}$ such that for any $1\leq i\leq m$, there exists $\lambda_{i}\in K^{\ast}$ such that $v_{i}^{n}=\lambda_{i}I_{n}$. It is worth observing that, since $v_{1}\cdots v_{m}\neq 0$, all  $\theta(i,j)$ satisfy the regular quantum commutative relations of Lemma \ref{regular quantum commutative.relations}. 

Now, suppose that $\dim R_{t}\geq 2$ for some $1\leq t\leq m$. Let $\{v,v_{t}\}$ be a linearly independent set in $R_{t}$. Given $1\leq j\leq m$ and $a_{j}\in R_{j}$, we have $a_{j}=\theta(j,t)v_{t}a_{j}v_{t}^{-1}$, i.e $v_{t}^{-1}a_{j}=\theta(j,t)a_{j}v_{t}^{-1}$. Hence, 
\begin{align*}
    (v_{t}^{-1}v)a_{j} &=\theta(t,j)(v_{t}^{-1}a_{j})v\\
    &=\theta(t,j)(\theta(j,t)a_{j}v_{t}^{-1})v\\
    &=a_{j}(v_{t}^{-1}v).
\end{align*}

Since $1\leq j\leq m$ are arbitrary, we conclude that $0\neq v_{t}^{-1}v\in Z(M_{n}(K))$, i.e., $v_{t}^{-1}v=\lambda I_{n}$, for some $\lambda\in K^{\ast}$. In this case, $v=\lambda v_{t}$, which is a contradiction. Consequently, $\dim R_{i}=1$, for all $1\leq i\leq m$.   In particular, if $R=R_{1}\oplus \cdots \oplus R_{m}$, then this is a regular quantum commutative decomposition of $R$ and $m=n^{2}$. 
\end{proof}

\begin{Corollary}
    \label{always.minimal}
Let $R=M_{n}(K)$ and $\mathscr{R}\colon R=R_{1}\oplus \cdots \oplus R_{m}$ be a regular quantum commutative decomposition of $R$. Then, $\mathscr{R}$ is minimal. In particular, $\mathscr{R}$ is a group-grading. 
\end{Corollary}

\begin{proof}
   By Lemma \ref{weaky.hipothesis}, $m=n^{2}$ and for all $1\leq i\leq m$ there exists $v_{i}\in R_{i}$ with $v_{i}^{n}\in K^{\ast}I_{n}$ and $R_{i}=\operatorname{Span}_{K}\{v_{i}\}$. Suppose by contradiction $\mathscr{R}$ is non-minimal, i.e there exist $1\leq i_{1},i_{2}\leq m$ with $i_{1}\neq i_{2}$ such that 
    \[
    \theta(i_{1},k)=\theta(i_{2},k),\mbox{ for all } 1\leq  k\leq m.
    \]
    Set $v:=v_{i_{1}}v_{i_{2}}^{-1}$, and let $1\leq k\leq m$. In this case, on the one hand we have \[
    v_{k}v_{i_{2}}=\theta(k,i_{2})v_{i_{2}}v_{k}\mbox{ that is } v_{k}=\theta(k,i_{2})\,(v_{i_{2}}v_{k}v_{i_{2}}^{-1}).
    \]
    On the other hand
    \begin{align*}
         vv_{k}=v_{i_{1}}v_{i_{2}}^{-1}v_{k} &= v_{i_{1}}v_{i_{2}}^{-1}(\theta(k,i_{2})v_{i_{2}}v_{k}v_{i_{2}}^{-1})\\
         &= \theta(k,i_{2})(v_{i_{1}}v_{k}v_{i_{2}}^{-1}) = (\theta(k,i_{2}) \theta(i_{1},k)) v_{k}v\\
         &= (\theta(k,i_{2})\theta(i_{2},k))v_{k}v=v_{k}v.
    \end{align*}
   Since $1\leq k\leq m$ is arbitrary, we conclude that $v\in Z(R)$ and so $v_{i_{1}}=\lambda v_{i_{2}}$, for some $\lambda\in K^{\ast}$, which implies that $0\neq v_{i_{1}}\in R_{i_{1}}\cap R_{i_{2}}$, which is a contradiction. Consequently, $\mathscr{R}$ is minimal and by Lemma \ref{result.m.dim(R)}, $\mathscr{R}$ is a group-grading.
\end{proof}

\begin{Theorem}
\label{realization.matrices}
Let $R=(\rR,\theta)$ be a regular quantum commutative algebra in the sense of Definition \ref{dQCD}, with $R=M_{n}(K)$. Then, $\mathscr{R}$ must be a $G$-regular grading with respect to a finite abelian group $G=(S,\star)$, $M_{G}^{R}=M^{R}$ and $\theta(i,j)$ is a bicharacter of $G$.  In particular, $G\cong \mathbb{Z}_{n_{1}}^{2}\times \cdots \times \mathbb{Z}_{n_{r}}^{2}$ where $n_{1}\cdots n_{r}=n$. 
\end{Theorem}
\begin{proof}
    Suppose that $m$ is the quantum length of $\mathscr{R}$. By Lemma \ref{weaky.hipothesis}, we have $m = n^{2} = \dim(R)$ and by Corollary \ref{always.minimal} $\mathscr{R}$ is minimal. Thus, the result follows from Corollary \ref{twisted.group} and \cite[Theorem 2.15]{Kochetov.book}. 
\end{proof}

\subsection{Non-simple regular quantum commutative algebras}\label{ssNSQCA}
\begin{Lemma}
\label{nilp.elements} Let $R=(\rR,\theta)$ be a finite-dimensional regular quantum commutative algebra with $\mathscr{R}$ minimal of quantum length $m$. Consider the Wedderburn--Malcev decomposition of $R$, $R = B+J(R)$, where $B = B_{1} \oplus \cdots \oplus B_{k}$ is the direct sum (as algebras) of simple subalgebras of $R$. Then, for any $1 \leq j \leq m$, there exists $x \in R_{j}$ such that $x = a + b$, where $a \in B$ and $b \in J(R)$, with $a \neq 0$. 
\end{Lemma}
\begin{proof}
    Since $R$ is finite dimensional, its Jacobson radical $J(R)$ is a nilpotent ideal. Let $q\in \mathbb{N}$ denote its nilpotency index. Suppose there exists $1\leq t\leq m$ such that all elements $x\in R_{t}$ belong to $J(R)$. Since $R$ is a regular quantum commutative algebra,  there exist $x_{i_{1}}$,\dots, $x_{i_{q}}\in R_{t}$ such that $x_{i_{1}}\cdots x_{i_{q}}\neq 0$ At the same time, $x_{i_{1}}\cdots x_{i_{q}} \in (J(R))^q=\{ 0\}$, which is a contradiction.  
\end{proof}

\begin{Remark}
\label{carrying.on.semissimple}
In particular, the elements $w_{1}\in R_{1}$,\dots, $w_{m}\in R_{m}$ of Lemma \ref{non.nilp} are of the form $w_{i}=\overline{w}_{i}+u_{i}$, where $r_{i}\in J(R)$ and $0\neq \overline{w}_{i}\in B$. Notice that 
\[
w_{i}w_{j}=\overline{w}_{i}\overline{w}_{j}+w,\quad w\in J(R)
\]
and
\[
w_{j}w_{i}=\overline{w}_{j}\overline{w}_{i}+w',\quad w'\in J(R),
\]
thus, since $w_{i}w_{j}=\theta(i,j)w_{j}w_{i}$, it follows that $0\neq \overline{w}_{i}\overline{w}_{j}=\theta(i,j)\overline{w}_{j}\overline{w}_{i}$. We observe that $\overline{w}_{1} \cdots \overline{w}_{m}$ is not nilpotent, because otherwise there exists $s \in \mathbb{N}$ such that $(w_{1} \cdots w_{m})^{s} \in J(R)$, which is a contradiction.

\end{Remark}
\begin{Lemma} 
\label{simple.component}
Let $R=(\rR,\theta)$ be a finite-dimensional regular quantum commutative algebra with $\mathscr{R}$ minimal of quantum length $m$. Consider the Wedderburn--Malcev decomposition of $R$, $R = B \oplus J(R)$, where $B = B_{1} \oplus \cdots \oplus B_{k}$ is the direct sum (as algebras) of  simple subalgebras of $R$. Then, there exists $1\leq \ell\leq k$ such that $B_{\ell}$ admits the structure of regular quantum commutative algebra of quantum length $m$, with the same regular quantum commutative factors as $R=(\rR,\theta)$.  
\end{Lemma}
\begin{proof} 
  For any $1\leq i\leq m$, $R_{i}$ has a basis $\{a_{i,1}+t_{i,1},\ldots, a_{i,n_{i}}+t_{i,n_{i}}\}$, where $a_{i,j}\in B $ and $t_{i,j}\in J(R)$, $1\leq j\leq n_{i}$. We set
  \[
  T_{i}=\operatorname{Span}_{K}\{a_{i,1},\ldots, a_{i,n_{i}}\},\:1\leq i\leq m.
  \]
   In this case, $B=T_{1}+ \cdots + T_{m}$. Let us consider $w_{1}$,\ldots, $w_{m}$, as in Lemma \ref{non.nilp}, and write $w_{i}=\overline{w_{i}}+v_{i}$, $\overline{w}_{i}\in B$ and $v_{i}\in J(R)$. By Remark \ref{carrying.on.semissimple}, for all $1\leq i,j\leq m$, we have:
   \begin{enumerate}
       \item[(i)] $\overline{w}_{i}\in T_{i}$;
       \item[(ii)] $\overline{w}_{1}\cdots \overline{w}_{m}$ is non-nilpotent;
       \item[(iii)] $b_{i}b_{j}=\theta(i,j)b_{j}b_{i}$, for all $b_{i}\in T_{i}$ and $b_{j}\in T_{j}$.
   \end{enumerate}

 Denoting by $e_{1}\in B_{1},\dots, e_{k}\in B_{k}$ the central orthogonal idempotent elements of $B$, for any $1\leq i\leq m$ we write 
\[
\overline{w}_{i}=u_{i,1}+\cdots +u_{i,k},\quad u_{i,t}\in B_{t},\quad 1\leq t\leq k,
\]
 in this way, we get 
\begin{equation}
\label{eq.S.C}
u_{i,t}u_{j,t}=\theta(i,j)u_{j,t}u_{i,t},\quad \text{for all}\quad 1\leq t\leq k.     
\end{equation}
Suppose that for all $1\leq t\leq k$, the element $h_{t}:=u_{1,t}\cdots u_{m,t}$ is nilpotent, say $h_{t}^{d_{t}}=0$ for some $d_{t}\in \mathbb{N}$. Given $s>d_{1}+\cdots +d_{k}$, by regularity there exists $\lambda\in K^{\ast}$ such that
\[
(\overline{w}_{1}\cdots  \overline{w}_{m})^{s}=\lambda H_{1}^{s}\cdots H_{k}^{s}=0
\]
which is a contradiction. We conclude there exists $1\leq \ell\leq k$ such that $h_{\ell}$ is non-nilpotent. Defining $C_{i}:=T_{i}e_{\ell}$, we have $u_{i,N}\in C_{i}$, for all $1\leq i\leq m$ and 
\[
B_\ell=Be_{\ell}=C_{1}+\cdots +C_{m}.
\]
 By Lemma \ref{weaky.hipothesis}, it follows that for any $1\leq i\leq m$, $u_{i,\ell}$ is invertible and  $C_{i}=\operatorname{Span}_{K}\{u_{i,\ell}\}$. 

Suppose there exists $2\leq j\leq m$ such that 
\[
u_{j,\ell}\in C_{j}\cap (C_{1}+\cdots+C_{j-1}).
\]
Then, there exist $\gamma_{1}$,\dots, $\gamma_{j-1}\in K$ such that $u_{j,\ell}=\gamma_{1}u_{1,\ell}+\cdots +\gamma_{j-1}u_{j-1,\ell}$. Without loss of generality, we can assume $\gamma_{d}\neq 0$, for all $1\leq d\leq j-1$. Now, for any $1\leq i\leq m$, we get
\begin{align*}
    u_{i,\ell}u_{j,\ell} &=\gamma_{1}u_{i,\ell}u_{1,\ell}+\cdots+\gamma_{j-1}u_{i,\ell}u_{j-1,\ell}\\
    &=\theta(i,1)(\gamma_{1}u_{1,\ell}u_{i,\ell})+\cdots +\theta(i,j-1)(\gamma_{j-1}u_{j-1,\ell}u_{i,\ell}).
\end{align*}
On the other hand,
\begin{align*}
   u_{i,\ell}u_{j,\ell} &=\theta(i,j)u_{j,\ell}u_{i,\ell}\\
   &= \theta(i,j)(\gamma_{1}u_{1,\ell}u_{i,\ell})+\cdots +\theta(i,j)(\gamma_{j-1}u_{j-1,\ell}u_{i,\ell}).
\end{align*}
In particular, $\theta(i,j)=\theta(i,1)$. Since $1\leq i\leq m$ is arbitrary, it follows that the columns $(\theta(i,1))_{1\leq i\leq m}$ and $(\theta(i,j))_{1\leq i\leq m}$ coincide, contradicting the minimality of $\mathscr{R}$. Therefore, by \cite[Lemma 6.6]{HOFFMAN}, $B_{\ell}=C_{1}\oplus \cdots \oplus C_{m}$. Hence, again by Lemma \ref{weaky.hipothesis},
\[
\mathscr{B}_{\ell}\colon B_{\ell}=C_{1}\oplus \cdots \oplus C_{m}
\]
is a regular quantum commutative decomposition for $B_{\ell}$. By Equation (\ref{eq.S.C}), the  regular quantum commutative factors are the same as in $R=(\rR,\theta)$. 
\end{proof}

We now observe that the first part of the proof of Lemma \ref{simple.component}, i.e., the part that ensures the existence of a simple subalgebra $B_{\ell}$ with $m$ nonzero components, does not use the minimality of $\mathscr{R}$. So if we combine this argument with Lemma \ref{weaky.hipothesis}, we obtain the following result.
\begin{Proposition}
\label{weaky.sum.simple}
Let $R=(\rR,\theta)$ be a finite dimensional regular quantum commutative algebra of quantum length $m$. Consider the Wedderburn--Malcev decomposition $R = B \oplus J(R)$, where $B = B_{1} \oplus \cdots \oplus B_{k}$ is the direct sum of  simple subalgebras of $R$. Then, given $1\leq i\leq k$, there exist indices $i_{1},\dots, i_{t}\in \{1,\ldots,m\}$ such that $B_{i}$ admits a decomposition
\[
B_{i}=D_{i_{1}}+\cdots+D_{i_{t}},
\]
where each $D_{i_j}$ is a subspace of $B_{i}$, satisfying:
\begin{enumerate}
    \item[(i)] For each $1\leq s\leq t$, there exists $0\neq u_{s}\in D_{i_{s}}$ such that $D_{i_{s}}=\operatorname{Span}_{K}\{u_{s}\}$ and $u_{s}^{q_{i}}=\mu_{s}1_{B_{i}}$ for some $\mu_{s}\in K^{\ast}$, where $q_{i}:=\sqrt{\dim B_{i}}$.
    \item[(ii)] For any $1\leq r,s \leq t$, $u_{r}u_{s}=\theta(i_{r},i_{s})u_{s}u_{r}$.
\end{enumerate}
Moreover, there exists $1\leq \ell\leq k$ and subspaces $C_{1}$,\ldots, $C_{m}$ of $B_{\ell}$ such that $B_{\ell}$ admits a decomposition $B_{\ell}=C_{1}+\cdots+C_{m}$ satisfying {\rm (i)} and {\rm (ii)}.
\end{Proposition}

The next result gives a positive answer to the Bahturin--Regev Conjecture \ref{Conjecture.Bahturin.Regev} for finite dimensional regular quantum commutative algebras over an algebraically closed field whose characteristic does not divide the quantum length of the respective algebra. 

\begin{Remark}
It is worth noting that the Bahturin--Regev Conjecture holds for finite dimensional commutative regular quantum commutative algebras, in view of Remark \ref{commutative.}.
\end{Remark}

\begin{Theorem}
\label{Bahturin.Regev}
Let $R=(\rR,\theta)$ be a finite dimensional regular quantum commutative algebra of quantum length $m$. Suppose $\operatorname{char}(K)\nmid m$.  Then, $\mathscr{R}$ is minimal if and only if $\det M^{R}\neq 0$. Moreover, in this case,  $m=\dim C$ where $C$ is a simple subalgebra of $R$ and $\det M^{R}=\pm m^{m/2}$.
\end{Theorem}
\begin{proof} We know that if $\det M^{R}\neq 0$, then $\mathscr{R}$ is minimal. Therefore we assume that $\mathscr{R}$ is minimal.  Consider the Wedderburn--Malcev decomposition of $R$, $R = B \oplus J(R)$, where $B = B_{1} \oplus \cdots \oplus B_{k}$ is the direct sum (as algebras) of simple subalgebras of $R$. 

By Lemma \ref{simple.component}, there exists $1\leq \ell\leq k$ such that  $C:=B_{\ell}$ admits a regular quantum commutative decomposition $\mathscr{C}$ which is minimal and has the same quantum decomposition matrix as $\mathscr{R}$, i.e., $M^{C}=M^{R}$. By Theorem \ref{realization.matrices}, $\mathscr{C}$ is a $G$-regular grading,  with $G$ a finite abelian group and $|G|=m=\dim C$ and $M_{G}^{C}=M^{C}=M^{R}$. It follows that the regular decomposition of $C$ as a $G$-graded regular algebra is minimal. Thus by Proposition \ref{Bahturin.Regev.Grading} we get $(\det M^{R})^{2}=(\det M_{G}^{C})^{2}=m^{m}$, i.e., $\det M^{R}=\pm m^{m/2}$. 
\end{proof}
 \begin{Corollary}  Let $R=(\rR,\theta)$ be a finite-dimensional regular quantum commutative algebra of quantum length $m$. Suppose $\operatorname{char}(K)=0$ and $\mathscr{R}$ is minimal. Then $m=\exp(R)$.
\end{Corollary}
\begin{proof}  By Theorem \ref{Bahturin.Regev}, $m=\dim C$, where $C$ is a simple subalgebra of $R$, and because $C$ is  central simple we have $m=\exp(C)$. Now, since $M^{C}=M^{R}$, by \cite[Theorem 3.1]{regevz2} we get $T(R)=T(C)$.  Consequently, $\exp(R)=\exp(C)=m$. 
\end{proof}

\section{Regular quantum commutative decompositions and set-gradings}\label{s6}

The restriction on the center in Lemma \ref{result.m.dim(R)} cannot be removed; that is, the minimality alone does not guarantee that a regular quantum commutative decomposition is a set grading.  
The following example illustrates a regular quantum commutative decomposition in a non-central algebra for which the decomposition is minimal but not a set grading.

\begin{Example}
\label{Example.non.set}
Let $R = M_{2}(K) \oplus M_{4}(K)$ and let $\xi= \sqrt{-1}$.  Consider $P = \operatorname{diag}(1,\xi,\xi^{2},\xi^{3}) \in M_{4}(K)$ and $D = \operatorname{diag}(1,-1) \in M_{2}(K)$, the Sylvester's clock matrices, and $Q = \begin{pmatrix}
0 & 1 & 0 & 0\\
0 & 0 & 1 & 0\\
0 & 0 & 0 & 1\\
1 & 0 & 0 & 0
\end{pmatrix} \in M_{4}(K)$ and $N = \begin{pmatrix}
0 & 1\\
1 & 0
\end{pmatrix} \in M_{2}(K)$, the shift matrices.  Define 
\[
R_{(0,0)}:=\operatorname{Span}_{K}\{(I_{2},0)\}\oplus \operatorname{Span}_{K}\{(0,I_{4})\},\quad R_{(2,0)}:=\operatorname{Span}_{K}\{(D,0)\}\oplus \operatorname{Span}_{K}\{(0,P^{2})\},
\]
\[
R_{(0,1)}:=\operatorname{Span}_{K}\{(N,0)\}\oplus \operatorname{Span}_{K}\{(0,Q)\},\quad R_{(2,1)}:=\operatorname{Span}_{K}\{(DN,0)\}\oplus \operatorname{Span}_{K}\{(0,P^{2}Q)\},
\]
and
\[
R_{(i,j)}:= \operatorname{Span}_{K}\{(0,P^{i}Q^{j})\},\quad \text{for all }\quad (i,j)\in \{0,1,2,3\}^{2}\setminus\{(0,0),(2,0),(0,1),(2,1)\}.
\]

Since the projections of $R_{(k,\ell)}$ onto $M_{2}(K)$ and $M_{4}(K)$ are the Pauli's grading respectively, it follows that 
\[
\mathscr{R}\colon R=\bigoplus_{(k,\ell)\in \{0,1,2,3\}^{2}} R_{(k,\ell)}
\]
is a regular quantum commutative decomposition of $R$ whose the quantum decomposition matrix is given by 
\[
M^{R}=(\xi^{jk-i\ell})_{((i,j),(k,\ell))}.
\]

 By \cite[Proposition 2.3]{bcommutation}, $\det M^{R}\neq 0$, thus  $\rR$ is minimal. On the other hand, $\mathscr{R}$ is not a set-grading because $(N^{2},0)=(I_{2},0)\in R_{(0,0)}$ but $(0,Q^{2})\in R_{(0,2)}$. 
\end{Example}

The following example presents a regular quantum commutative decomposition that is a set-grading, but whose grading is not realized as a group-grading.

\begin{Example}
\label{Example.length}
In $M_{6}(K)$, consider the following matrices $L=\operatorname{diag}(D,D,D)$ and $J=\operatorname{diag}(0,0,N)$, where $D$ and $N$ are as Example \ref{Example.non.set}. These matrices satisfy the following relations
\begin{equation}
\label{def.equation}
 L^{2}=I_{6},\quad    J^{2}=\operatorname{diag}(0,0,I_{2}),\quad \text{and}\quad  LJ=-JL.
\end{equation}

Therefore if $R$ is the subalgebra of $M_{6}(K)$ generated by $L$, and $J$, it can be seen that the set $\{I_{6},J,J^{2},L,LJ,LJ^{2}\}$ is a basis of $R$. Now, we can write
\[
\mathscr{R}\colon \quad R=\operatorname{Span}_{K}\{I_{6}\}\oplus  \operatorname{Span}_{K}\{J^{2}\}\oplus \operatorname{Span}_{K}\{LJ\}\oplus \operatorname{Span}_{K}\{LJ^{2}\}\oplus \operatorname{Span}_{K}\{J\}\oplus \operatorname{Span}_{K}\{L\}.
\]
Notice $(I_{6})(J^{2})(LJ)(LJ^{2})(J)L=\operatorname{diag}(0,0,-D)$, which is non-nilpotent. Therefore, by the relations (\ref{def.equation}), it follows that $\mathscr{R}$ is a regular quantum commutative decomposition of $R$ with quantum decomposition matrix given by
\[
M^{R}=\begin{pmatrix}
1 & 1 & 1 & 1 & 1 & 1  \\
1 & 1 & 1 & 1 & 1 & 1 \\
1 & 1 & 1 & -1 & -1 & -1 \\
1 & 1 & -1 & 1 & -1 & 1\\
1 & 1 & -1 & -1 & 1 & -1\\
1 & 1 & -1 & 1 & -1 & 1
\end{pmatrix}.
\]
In particular, $\mathscr{R}$ is not minimal. Observe that the subalgebra generated by $D$ and $N$ is isomorphic to $M_{2}(K)$. Since $LJ^{2}=\operatorname{diag}(0,0,D)$, it follows that the subalgebra $\langle J,LJ^{2}\rangle$ is also isomorphic to $M_{2}(K)$. Now, let $C:=\operatorname{diag}(D,D,0)$. Then $C^{2}=\operatorname{diag}(I_{2},I_{2},0)$, which we identify with the $4\times 4$ identity matrix. Moreover,
\[
L=C+LJ^{2},\quad C(LJ^{2})=0,\quad \text{and}\quad I_{6}=C^{2}+(LJ^{2})^{2}.
\]
On the other hand, if $\mathcal{V}$ is the subalgebra generated by $C$, then $\mathcal{V}\cong K[x]/(x^{2}-1)\cong K\oplus K$. Hence,
\[
R\cong K\oplus K\oplus M_{2}(K).
\]

Of course, $\mathscr{R}$ is a set-grading. Suppose it is realized by a $G$-grading, where $G$ is a finite group, i.e. there exist $g_{1}$,\dots, $g_{6}\in G$ such that 
\[
R_{g_{1}}=\operatorname{Span}_{K}\{I_{6}\},\quad R_{g_{2}}=\operatorname{Span}_{K}\{J^{2}\},\quad R_{g_{3}}=\operatorname{Span}_{K}\{LJ\}
\]

\[
R_{g_{4}}= \operatorname{Span}_{K}\{LJ^{2}\},\quad R_{g_{5}}=\operatorname{Span}_{K}\{J\},\quad \text{and}\quad R_{g_{6}}=\operatorname{Span}_{K}\{L\}.
\]

Using (\ref{def.equation}), we get
\[
g_{2}g_{4}=g_{4},\quad g_{1}g_{4}=g_{4}
\]
i.e $g_{2}g_{4}=g_{1}g_{4}$. Thus, $g_{2}=g_{1}$ which is a contradiction. 
\end{Example}

\begin{Theorem}
\label{group.grading.matrices.algebra}
Let $R=M_{n}(K)$ and $\rR\colon R=R_{1}+ \cdots + R_{l}$ be a decomposition of $R$ as (not
necessarily direct) a  sum of vector spaces satisfying the following:
\begin{enumerate}
    \item[(i)] $R_{i}\neq 0$, for all $1\leq i\leq l$;
    \item[(ii)] for all $1\leq i,j\leq l$, there exists a root of unity $\theta(i,j)\in K^{\ast}$ such that
\[
a_{i}a_{j}=\theta(i,j)a_{j}a_{i}
\]
for all $a_{i}\in R_{i}$ and $a_{j}\in R_{j}$.
\end{enumerate}

Then we have the following possibilities:
\begin{enumerate}
    \item[(a)] If $\rR$ is the direct sum, then $l=n^{2}$ and $\rR$ is a regular group grading by a finite abelian group.
    \item[(b)] If $\rR$ is not a direct sum, then there exists $\{i_{1},\ldots, i_{q}\}$, a subset of $\{1,\ldots, l\}$, such that:
    \begin{enumerate}
        \item[(b.1)]  For all $j\in \{1,\ldots,l\}\setminus\{i_{1},\ldots, i_{q}\}$,   $R_{j}=R_{i_{k}}$ for some $1\leq k\leq q$.
        \item[(b.2)] $\rR'\colon R=R_{i_{1}}\oplus \cdots \oplus R_{i_{q}}$, $q=n^{2}$  and $\rR'$ is a regular group grading by a finite abelian group.
    \end{enumerate}
\end{enumerate}
 Moreover, in each case, if $\mathcal{P}:=(\theta(i,j))_{1\leq i,j\leq l}$ and $\mathcal{P}_{1}$,\dots, $\mathcal{P}_{l}$ denotes the rows of $\mathcal{P}$, we have
 \[
 n^{2}=\dim R=\# \{\mathcal{P}_{1},\ldots, \mathcal{P}_{l}\}.
 \]
\end{Theorem}
\begin{proof}
    Suppose $\rR$ is a direct sum. Then, by Lemma \ref{weaky.hipothesis}, $R = R_{1} \oplus \cdots \oplus R_{l}$ is a regular quantum commutative decomposition of $R$, with $l= n^{2}$, and by Theorem \ref{realization.matrices}, $\mathscr{A}$ is a regular group-grading by a finite abelian group, which proves (a).
    
     Let us proof (b). By Lemma \ref{weaky.hipothesis}, we know that for all $1\leq i\leq l$, there exists an invertible element $v_{i}\in R_{i}$ such that $R_{i}=\operatorname{Span}_{K}\{v_{i}\}$.  Since $\rR$ is not a direct sum, there exists $2\leq t\leq l$ such that $v_{t}\in R_{t}\cap (R_{1}+\cdots +R_{t-1})$, and we can write $v_{t}=\gamma_{1}v_{1}+\cdots+\gamma_{t-1} v_{t-1}$. Without loss of generality we assume $\gamma_{1}\neq 0$. Now, for any $1\leq i\leq l$, we get
\begin{align*}
    v_{i}v_{t} &=v_{i}(\gamma_{1}v_{1})+\cdots+v_{i}(\gamma_{t-1}v_{t-1})\\
    &=\theta(i,1)(\gamma_{1}v_{1}v_{i})+\cdots +\theta(i,t-1)(\gamma_{t-1}v_{t-1}v_{i}).
\end{align*}
On the other hand,
\begin{align*}
   v_{i}v_{t} &=\theta(i,t)v_{t}v_{i}\\
   &= \theta(i,t)(\gamma_{1}v_{1}v_{i})+\cdots +\theta(i,t)(\gamma_{t-1}v_{t-1}v_{i}).
\end{align*}
Since $v_{t}$ is invertible we obtain $\theta(i,t)=\theta(i,1)$. Hence, we have shown that
\[
\theta(i,t)=\theta(i,1),\quad \text{for all}\quad 1\leq  i\leq l.
\]
Thus, it can be seen that $v_{t}v_{1}^{-1}\in Z(R)$, and it follows that $v_{t}=\lambda v_{1}$, for some $\lambda\in K^{\ast}$, i.e $R_{t}=R_{1}$. Therefore, identifying equal components, we conclude that there exists a subset $\{i_{1},\ldots, i_{q}\}\subseteq \{1,\ldots, l\}$ such that for any $j\in \{1,\ldots,l\}\setminus\{i_{1},\ldots, i_{q}\}$ there exists $1\leq k\leq q$ with $R_{j}=R_{i_{k}}$, which proves (b.1). Item (b.2) follows directly from item (a). Finally, consider the following relation equivalence in $\{1,\ldots,l\}$:
\[
i_{1}\sim i_{2}\quad \text{if and only if}\quad \theta(k,i_{1})=\theta(k,i_{2}),\quad \text{for all $1\leq k\leq l$}.
\]
 Then, if $\operatorname{\textbf{cl}}(i)$ denotes the equivalence class of $i$ with respect to $\sim$, there exists $q\leq l$ such that $\{1,\ldots,l\}=\operatorname{\textbf{cl}}(i_{1})\cup \cdots \cup \operatorname{\textbf{cl}}(i_{q})$ (disjoint union). Given $1\leq j\leq q$ and $t\in \operatorname{\textbf{cl}}(i_{j})$ we have $\mathcal{P}_{t}=\mathcal{P}_{i_{j}}$ and by the above calculations we have $R_{t}=R_{i_{j}}$. Thus by item (b.2)
 \[
 \dim R=q=\#\{\mathcal{P}_{1},\ldots, \mathcal{P}_{l}\}.
 \]
 
\end{proof}

\begin{Remark}
    
    \label{redution.}
         The minimality assumption cannot be removed from the hypothesis of Lemma \ref{simple.component}. As an illustration, in Example \ref{Example.length} we have a regular quantum commutative decomposition of length $6$; however, the length of the induced regular quantum commutative decomposition on $M_{2}(K)$ is $4$, and on the simple components isomorphic to $K$ it is exactly $1$.
\end{Remark}

\begin{Lemma}
\label{grading}
Let $R=(\rR,\theta)$ be a finite dimensional regular quantum commutative algebra with 
$R=B_{1}\oplus B_{2}$ (a direct sum of ideals), where $B_{1}$ and $B_{2}$ are the simple components of $R$. Suppose that $\mathscr{R}$ is a set-grading and $\mathscr{R}$ is minimal. Then, $\mathscr{R}$ is a group grading. 
 \end{Lemma}
\begin{proof} Let $e_{1}$ and $e_{2}$ be the orthogonal idempotents associated to the decomposition $R=B_{1}\oplus B_{2}$ and $\mathscr{R}\colon R=R_{1}\oplus \cdots \oplus R_{m}$ be the regular quantum commutative decomposition of $R$. By hypothesis, given $1\leq i,j\leq m$, there exists $f(i,j)\in \{1,\ldots, m\}$ such that $R_{i}R_{j}\subseteq R_{f(i,j)}$. We observe that, since $e_{1}$, $e_{2}$ are central in $R$, it follows that for any $1\leq i,j\leq m$ and $k\in \{1,2\}$ 
\begin{equation}
\label{eq.continuity}
    (R_{i}e_{k})(R_{j}e_{k})\subseteq R_{f(i,j)}e_{k}.
\end{equation}

By Lemma \ref{result.m.dim(R)}, we can assume $\mathscr{B}_{1}\colon B_{1}=R_{1}e_{1}\oplus \cdots \oplus R_{m}e_{1}$ is the regular quantum commutative decomposition of $B_{1}$. Thus, by Theorem \ref{group.grading.matrices.algebra}, there exists $\star\colon S\times S\rightarrow S$ be a binary operation on $S=\{1,\ldots,m\}$ such that $G=(S,\star)$ is a finite abelian group and $\mathscr{B}_{1}$ is a $G$-regular grading and we can assume $s_{0}$ is the neutral element of $G$.  We need to show that $\star$ coincides with $f$. 

Indeed, suppose there exist 
$i$, $j\in S$ such that $i\star j\neq f(i,j)$.  Take $0\neq x\in R_{i}e_{1}$ and $0\neq y\in R_{j}e_{1}$ such that $xy\neq 0$. On one hand, we have $xy\in R_{i\star j}e_{1}$; on the other hand, by  (\ref{eq.continuity}), it follows that $xy\in R_{f(i,j)}e_{1}$.  Thus, $0\neq xy\in R_{f(i,j)}e_{1}\cap R_{i\star j}e_{1}$, which is a contradiction since $R_{f(i,j)}e_{1}$ is a direct summand of $B_{1}$. Therefore, $f= \star$ and $\mathscr{B}_{1}$ is a group-grading with respect to the finite abelian group $G=(S,f)$.

Now, let $i_{1}$,\dots, $i_{q}\in \{1,\ldots,m\}$ such that $\mathscr{B}_{2}\colon B_{2}=R_{i_{1}}e_{2}\oplus \cdots \oplus R_{i_{q}}e_{2}$. Then, by Theorem \ref{group.grading.matrices.algebra}, if $T=\{i_{1},\ldots, i_{q}\}$, then there exists $\circ\colon T\times T\rightarrow T$ such that $H=(T,\circ)$ is a finite abelian group and $\mathscr{B}_{2}$ is a regular grading of quantum length $q$. Since $\mathscr{B}_{2}$ is a regular grading with respect to $H$, by (\ref{eq.continuity}) we have $f(T,T)\subseteq T$ and $\circ=f\mid_{ T\times T}$. Assume $i_{1}$ is the neutral element of $H$.  Then, $i_{1}\circ i_{2}=i_{2}$, i.e, $f(i_{1},i_{2})=i_{2}$; since $T\subseteq S$, it follows that $i_{1}\star i_{2}=i_{2}$. On the other hand, $s_{0}\star i_{2}=i_{2}$ which implies $s_{0}=i_{1}$. Consequently, $H$ is a subgroup of $G$ and $\mathscr{R}$ is a $G$-grading. 
\end{proof}

By applying induction to the result of Lemma \ref{grading}, together with Theorem \ref{realization.matrices}, we obtain the following.

\begin{Theorem}
	\label{Set.grading.semissimple}
	 Let $R$ be a finite-dimensional semisimple algebra with a set grading $\mathscr{R}\colon R=R_{1}\oplus \cdots \oplus R_{m}$. Suppose $\mathscr{R}$ satisfies the following conditions:
\begin{enumerate}
    \item[(i)] $\mathscr{R}$ is a regular quantum commutative decomposition of $R$.
    \item[(ii)] $\mathscr{R}$ is minimal.
\end{enumerate}
Then, $\mathscr{R}$ is a group-grading by a finite abelian group $G$. Moreover, $m=\dim B_{\ell}$, where $B_{\ell}$ is a simple component of $R$ and $G\cong \mathbb{Z}_{n_{1}}^{2}\times \cdots \times \mathbb{Z}_{n_{r}}^{2}$, with $n_{1}\cdots n_{r}=\sqrt{\dim B_{\ell}}$.
    
\end{Theorem}

\section{A Necessary condition for the existence of a regular quantum commutative decomposition}\label{s7}

\begin{Proposition} 
\label{necessary.condition}
Let $R=(\rR,\theta)$ be a finite dimensional regular quantum commutative algebra. Consider its Wedderburn--Malcev decomposition, $R = B +J(R)$, where $B = B_{1} \oplus \cdots \oplus B_{k}$ is the direct sum (as algebras) of simple subalgebras of $R$. If $m$ is the quantum length of $\mathscr{R}$, and $q_{i}:=\sqrt{\dim B_{i}}$, $1\leq i\leq k $, then
\begin{enumerate}
    \item[(i)] There exists $1\leq \ell\leq k$, such that
\[
q_{s}\leq q_{\ell}\leq \sqrt{m},\quad \text{for all}\quad 1\leq s\leq k.
\]    
\item[(ii)] For any $1\leq s\leq k$ such that $1<q_{s}<q_{\ell}$, we have $\operatorname{gcd}(q_{s},q_{\ell})>1$. 
\end{enumerate}
\end{Proposition}

	\begin{proof} Denote by $e_{1}$,\dots, $e_{k}$ the orthogonal idempotents associated with the decomposition $B_{1}\oplus \cdots \oplus B_{k}$.  By Proposition \ref{weaky.sum.simple}, there exists $1\leq \ell\leq k$ such that $B_{\ell}$ admits a decomposition of the form 
    \[
    B_{\ell}=C_{1}+\cdots+C_{m}
    \]
    with nonzero vector spaces $C_{1}$,\dots, $C_{m}$ satisfying: For all $1\leq i,j\leq m$
    \begin{enumerate}
        \item[(1)] $C_{i}=\operatorname{Span}_{K}\{u_{i}\}$, where $u_{i}^{q_{\ell}}=\mu_{i}1_{B_{\ell}}$, for some $\mu_{i}\in K^{\ast}$.
        \item[(2)] $u_{i}u_{j}=\theta(i,j)u_{j}u_{i}$.
    \end{enumerate}
     In particular, all $\theta(i,j)$ are $q_{\ell}$-th roots of unity, and by Theorem \ref{group.grading.matrices.algebra} we have $q_{\ell}\leq \sqrt{m}$.
     
     Now, take $1\leq s\leq k$ with $s\neq \ell$.  Again by Proposition \ref{weaky.sum.simple}, there exists $t\leq m$, indices $i_{1}$,\ldots, $i_{t}\in \{1,\ldots, m\}$ and nonzero vector subspaces $D_{i_{1}}$,\ldots, $D_{i_{t}}$ such that  
		\[
		B_{s}=D_{i_{1}}+\cdots+D_{i_{t}}
		\]
        and
        \[
        d_{i_{j}}d_{i_{l}}=\theta(i_{j},i_{l})d_{i_{l}}d_{i_{j}}
        \]
         for all $1\leq j,l\leq t$ and all $d_{i_{j}}\in D_{i_{j}}$, $d_{i_{l}}\in D_{i_{l}}$. 
         
         Now, consider the matrices $\mathcal{P}:=(\theta(i,j))_{1\leq i,j\leq m}$ and $\mathcal{L}:=(\theta_{i_{j},i_{l}})_{1\leq j,l\leq t}$.  By Theorem \ref{group.grading.matrices.algebra}, if $\mathcal{P}_{1}$,\ldots, $\mathcal{P}_{m}$ and $\mathcal{L}_{1}$,\ldots, $\mathcal{L}_{t}$ denotes the rows of $\mathcal{P}$ and $\mathcal{L}$ respectively, we have
        \begin{equation}
        \label{equal.rows.1}
        q_{\ell}^{2}=\#\{\mathcal{P}_{1},\ldots, \mathcal{P}_{m}\} ,\quad \text{and}\quad q_{s}^{2}= \#\{\mathcal{L}_{1},\ldots, \mathcal{L}_{t}\}.
    \end{equation}
    Notice the following: Given $\mathcal{L}_{j}=(\theta(i_{j},i_{l}))_{1\leq l\leq t}$ and  $\mathcal{L}_{k}=(\theta(i_{k},i_{l}))_{1\leq l\leq t}$, if $\mathcal{P}_{i_{j}}=\mathcal{P}_{i_{k}}$, then we have 
    \[
    \mathcal{L}_{j}=\mathcal{L}_{k}.
    \]
    Hence, we get
     \begin{equation}
      \label{eq.equal.rows}
   \#\{\mathcal{L}_{1},\ldots, \mathcal{L}_{t}\}\leq \#\{\mathcal{P}_{1},\ldots, \mathcal{P}_{m}\}.
    \end{equation}
    i.e, $q_{s}\leq q_{\ell}$. The equality holds   if $t=m$, because in this case $\mathcal{P}=\mathcal{L}$. We have shown that
    \[
    q_{s}\leq q_{\ell}\leq \sqrt{m},
    \]
     which proves (i).

    Now, suppose $1<q_{s}<q_{\ell}$. By Lemma \ref{weaky.hipothesis}, for any $1\leq j\leq t$, $D_{i_{j}}=\operatorname{Span}_{K}\{\widetilde{u}_{j}\}$  where $\widetilde{u}_{j}$ satisfies $\widetilde{u}_{j}^{q_{s}}=\lambda_{j}1_{B_{s}}$, for some $\lambda_{j}\in K^{\ast}$. In particular,  all $\theta(i_{j},i_{\ell})$ are $q_{s}$-th roots of unity. Since we already know that all  $\theta(i_{j},i_{l})$ are $q_{\ell}$-th roots of unity, we conclude $\operatorname{gcd}(q_{s},q_{\ell})>1$. 
	\end{proof}   
\begin{Remark}
	Consider the algebra $R:=M_{4}(K)\oplus M_{6}(K)$. If it admitted a regular quantum decomposition, then, without loss of generality, we would be able to assume that it is minimal, and thus, via projection, we would obtain regular gradings on $M_{4}(K)$ and on $M_{6}(K)$ in which the values of the bicharacter of the decomposition of $M_{4}(K)$ are obtained from the values of the bicharacter of the decomposition of $M_{6}(K)$. However, using the classification of bicharacters (see \cite[Remark 2.16]{Kochetov.book}), one can show that it is impossible to obtain all the values of a bicharacter of $M_{4}(K)$ from those of a bicharacter of $M_{6}(K)$. We leave the details of the proof to the reader. Thus, the result concerning the greatest common divisor in Proposition \ref{necessary.condition}, although necessary, is not sufficient to guarantee the existence of a regular quantum decomposition in an algebra.
\end{Remark}

\begin{Corollary} Let $R$ be an algebra with two simple components $B$ and $C$ such that $B\cong M_{q}(K)$ and $C\cong M_{p}(K)$, with $\operatorname{gcd}(p,q)=1$. Then $R$ does not admit a regular quantum commutative decomposition.
\end{Corollary}

Assuming $\operatorname{char}(K)=0$, we know that if two algebras admit regular quantum commutative decompositions with the same quantum length and same quantum decomposition matrix, then their $T$-ideals coincide (see \cite[Theorem 3.1]{regevz2}). The above corollary shows that the equality of $T$-ideals is not sufficient to guarantee that one of the algebras admits a regular regular quantum commutative decomposition (in particular, it does not even admit a regular grading). For instance, take $R:=M_{q}(K)\oplus M_{p}(K)$, with $\operatorname{gcd}(p,q)=1$ and $q>p$. Then $T(R)=T(M_{p}(K))$, but $R$ does not admit a regular quantum commutative decomposition.

The next example shows that there exist algebras which do not admit a regular quantum commutative decomposition, but do admit a regular grading by a non-abelian group. 

\begin{Example} Suppose $\operatorname{char}(K)=0$. Consider the symmetric group $S_{5}$ and the trivial cocycle $\alpha\equiv 1$. Let $R:=K^{\alpha}S_{5}\cong KS_{5}$. It is well known that the degrees of the irreducible representations of $S_{5}$ in the non-decreasing order are: $1$,$1$,$4$,$4$,$5$,$5$,$6$. Thus, $KS_{5}\cong K^{\oplus 2}\oplus M_{4}(K)^{\oplus 2}\oplus M_{5}(K)^{\oplus 2}\oplus M_{6}(K)$. It follows that $A$ does not admit a regular regular quantum commutative  decomposition, but $A$ is a $S_{5}$-graded regular algebra with trivial bicharacter. 
\end{Example}

The next example presents an algebra that admits neither a regular quantum commutative structure nor a structure of a $G$-graded regular algebra with $G$ a finite non-abelian group.

\begin{Example} Suppose $\operatorname{char}(K)=0$. Consider $R:=M_{3}(K)\oplus M_{2}(K)$ which does not admit a regular quantum commutative decomposition. In particular, it does not have a structure of a regular grading by any finite abelian group. Suppose $R$ is a $G$-graded regular algebra where $G$ is a finite non-abelian group with $|G|\leq \dim R=13$. If $R\cong K^{\alpha}G$, $\alpha\in H^{2}(G,K^{\ast})$, then by \cite[Lemma 1.4.4]{Karpilovsky.vol.3}, $|G|=\dim M_{3}(K)+\dim M_{2}(K)=\dim R=13$, which is a contradiction because in this case, $G\cong \mathbb{Z}_{13}$. Thus, in light of Proposition \ref{Proposition.KG}, $K^{\alpha}G\cong M_{3}(K)$, for some $\alpha\in H^{2}(G,K^{\ast})$, which is again a contradiction because in this case, by \cite[Theorem 2.15]{Kochetov.book}, $G\cong \mathbb{Z}_{3}\times \mathbb{Z}_{3}$. We conclude that $R$ is not a $G$-graded regular algebra for any finite non-abelian group $G$.
\end{Example}

\begin{Proposition}
	\label{divisors}
		Let $n_{1}$, $n_{2}\in \mathbb{N}$ with $n_{1}\mid n_{2}$. Then $M_{n_{1}}(K)\oplus M_{n_{2}}(K)$ admits a regular grading by a finite abelian group.
\end{Proposition}
	\begin{proof}
		Since $n_{1}\mid n_{2}$, we can write $n_{2}=n_{1}q$ for some $q\in \mathbb{N}$. Then 
		\[
		R\cong M_{n_{1}}(K)\otimes \big(K\oplus M_{q}(K)\big).
		\]
		Consider the Pauli's grading on $M_{q}(K)$:
		\[
		M_{q}(K)=\bigoplus_{(k,l)\in \mathbb{Z}_{q}\times \mathbb{Z}_{q}} B_{(k,l)}.
		\]
		In this way, if we define $C_{(0,0)}:=K\oplus B_{(0,0)}$ and $C_{(k,l)}:= B_{(k,l)}$ for $(k,l)\neq (0,0)$, it follows immediately that $K\oplus M_{q}(K)=\bigoplus_{(k,l)\in \mathbb{Z}_{q}\times \mathbb{Z}_{q}} C_{(k,l)}$ is a regular grading. Now, set
		\[
		T:=(\mathbb{Z}_{n_{1}}\times \mathbb{Z}_{n_{1}})\times  (\mathbb{Z}_{q}\times \mathbb{Z}_{q})
		\]
		and consider the Pauli's grading on $M_{q}(K)$:
		\[
		M_{n_{1}}(K)=\bigoplus_{(i,j)\in \mathbb{Z}_{n_{1}}\times \mathbb{Z}_{n_{1}}} L_{(i,j)}.
		\]
		For any $((i,j),(k,l))\in T$ we define $R_{((i,j),(k,l))}:=L_{(i,j)}\otimes C_{(k,l)}$, then, by \cite[Lemma 26]{Eli1} (or \cite[Proposition 2.1]{bahturin2009graded}), it follows that
		\[
		R=\bigoplus_{((i,j),(k,l))\in T} R_{((i,j),(k,l))}
		\]
		 is a regular grading on $R$, whose decomposition matrix is the Kronecker product of the decomposition matrices  of $M_{n_{1}}(K)$ and $K\oplus M_{q}(K)$.
	
	\end{proof}

\begin{Corollary} 
	\label{p.group}
	Let $p>1$ be a prime number and $R=M_{n_{1}}(K)\oplus \cdots \oplus M_{n_{r}}(K)$ be the direct sum of matrix algebras such that for all $1\leq i\leq r$, $n_{i}=p^{\ell_{i}}$, for some $\ell_{i}\in \mathbb{N}$. Then $R$ admits a regular grading by some finite abelian group $T$.   
\end{Corollary}

\begin{proof}
	We proceed by induction on $r\geq 1$. If $r=1$, then $R$ admits a regular quantum commutative decomposition given by the Pauli's grading. If $r=2$ then the result follows from Proposition \ref{divisors}.
	
	Assume $r>2$, and suppose that for any $2\leq s<r$, any algebra $W$ of the form 
	\[
	W=M_{p^{k_{1}}}(K)\oplus \cdots \oplus M_{p^{k_{s}}}(K),
	\]
	with $k_{1}$,\dots, $k_{s}\in \mathbb{N}$, admits a regular grading by a finite abelian group. 
	
	Without loss of generality, assume that $\ell_{1}\leq \ell_{i}$ for all $1\leq i\leq r$. Then
	\[
	R=M_{p^{\ell_{1}}}(K)\oplus \cdots \oplus M_{p^{\ell_{r}}}(K)\cong M_{p^{\ell_{1}}}(K)\otimes \big(K\oplus B\big),
	\]
	where 
	\[
	B:=M_{p^{\ell_{2}-\ell_{1}}}(K)\oplus \cdots \oplus M_{p^{\ell_{r}-\ell_{1}}}(K).
	\]
	By the induction hypothesis, there exists a finite abelian group $H$ with neutral element $e\in H$ such that 
	\[
	B=\bigoplus_{h\in H} B_{h}
	\]
	is a regular grading of $B$. Define
	\[
	C_{e}:=K\oplus B_{e},\quad C_{h}:=B_{h}\quad \text{for all $h\in H$, $h\neq e$}.
	\]
	It follows that $K\oplus B=\bigoplus_{h\in H} C_{h}$ is a regular grading. Now, consider the Pauli's grading on $M_{p^{\ell_{1}}}(K)$:
	\[
	M_{p^{\ell_{1}}}(K)=\bigoplus_{(i,j)\in \mathbb{Z}_{p^{\ell_{1}}}\times \mathbb{Z}_{p^{\ell_{1}}}} L_{(i,j)}.
	\]
	Let $T:=(\mathbb{Z}_{p^{\ell_{1}}}\times \mathbb{Z}_{p^{\ell_{1}}})\times H$, and for any $\mathbf{t}=((i,j),h)\in T$ define 
	\[
	R_{\mathbf{t}}:=L_{(i,j)}\otimes C_{h}.
	\]
	Then, as in Proposition \ref{divisors} we conclude that $R=\bigoplus_{\mathbf{t}\in T}R_{\mathbf{t}}$ is a regular grading whose decomposition matrix is the Kronecker product of the decomposition matrices  of $M_{p^{\ell_{1}}}(K)$ and $K\oplus B$.
	\end{proof}

\begin{Remark}
	In \cite[Corollary 4.1]{Bahturin.Parmenter}, it is shown that if $G$ is a finite group and $K$ is an algebraically closed field with $\operatorname{char}(K)=0$ or $\operatorname{char}(K)\nmid |G|$, then the group algebra $KG$ admits a $\beta$-commutative grading for a suitable bicharacter $\beta$ defined on a finite abelian group $Q$. However, such gradings are not necessarily regular, since the second regularity condition may fail.  For instance, suppose $\operatorname{char}(K)=0$, and let $R=K\oplus Ku$ where $u^{2}=0$. Then $R$ becomes a $\mathbb{Z}_{2}$-graded algebra with $R_{0}=K$ and $R_{1}=Ku$, and it is $\beta$-commutative, where $\beta\colon \mathbb{Z}_{2}\times \mathbb{Z}_{2}\rightarrow K^{\ast}$ is given by $\beta(0,0)=\beta(1,0)=\beta(0,1)=1$ and $\beta(1,1)=-1$. However, $R$ is not regular since $u^{2}=0$. 

The next result shows that, under the same assumptions on $K$, one can in fact guarantee the existence of regular gradings on $KG$ when $G$ is a $p$-group for some prime number $p>1$.
\end{Remark}

Given a finite group $G$ with neutral element $e\in G$, we define the \empty{upper central series} of $G$ as the sequence of normal subgroups of $G$:
\[
Z_{0}\subseteq Z_{1}\subseteq\cdots \subseteq Z_{n}\subseteq\cdots  
\]
where $Z_{0}:=\{e\}$, $Z_{1}:=Z(G)$, and for $n\geq 1$, $Z_{n}$ is defined inductively by the condition $Z_{n+1}/Z_{n}=Z(G/Z_{n})$. In other words, if $\pi_{n}\colon G\rightarrow G/Z_{n}$ is the natural projection, then $Z_{n+1}:=\pi_{n}^{-1}(Z(G/Z_{n}))$. We say that $G$ is \empty{nilpotent} if there exists $n\in \mathbb{N}$ such that $Z_{n}=G$.  For instance, it is well known that any group of order $|G|=p^{n}$ (where $p>1$ is a prime number) is nilpotent.

\begin{Remark}
	\label{normal.abelian.non-central}
	Let $G$ be a nilpotent non-abelian group with neutral element $e\in G$. Then there exists a normal abelian subgroup $H$ of $G$ such that $H\nsubseteq Z_{1}$. Indeed, since $G$ is nilpotent and non-abelian, we have $Z_{1}\varsubsetneq Z_{2}$, and thus there exists $x\in Z_{2}\setminus Z_{1}$. Consider the subgroup generated by $\{x\}\cup Z_{1}$. Clearly, $H$ is abelian and $H\nsubseteq Z(G)$. Moreover, for any $g\in G$, we have $(xZ_{1})(gZ_{1})=(gZ_{1})(xZ_{1})$, that is, $gxg^{-1}\in xZ_{1}\subseteq H$. This shows that $H$ is a normal subgroup of $G$, and the result follows.  
\end{Remark}

\begin{Lemma}
	\label{div.|G|}
	Let $G$ be a nilpotent non-abelian group with neutral element $e\in G$ and suppose that either $\operatorname{char}(K)=0$ or $\operatorname{char}(K)\nmid |G|$. Then, for any irreducible $KG$-module $V$, we have $\dim(V)\mid |G|$.  
\end{Lemma}
\begin{proof}
	
	By \cite[Theorem 5.1.9]{herstein}, we may assume that $\operatorname{char}(K)\nmid |G|$. We prove the result by induction on $|G|$. The result is clear if $|G|=1$. 
	
	Assume that for any nilpotent group $Q$ with $1\leq |Q|<|G|$, we have $\dim(W)\mid |Q|$ whenever $W$ is an irreducible $KQ$-module. Let $V$ be an irreducible $KG$-module. By Blichfeldt's theorem (see \cite[Corollary 50.7]{Curtis.Reiner}), there exists a normal subgroup $\widetilde{H}$ of $G$ and an irreducible $K\widetilde{H}$-module $W$ such that 
	\[
	V\cong KG\otimes_{K\widetilde{H}} W,\quad \text{(as $KG$-modules)}.
	\]
	Let $\mathscr{T}:=\{g_{1},\ldots, g_{[G:\widetilde{H}]}\}$ be a transversal of $\widetilde{H}$ in $G$. Then
	\[
	V\cong KG\otimes_{K\widetilde{H}} W\cong \bigoplus_{i=1}^{[G:\widetilde{H}]}(g_{i}\otimes_{K\widetilde{H}} W).
	\]
	Therefore,
	\[
	\dim(V)=[G:\widetilde{H}]\dim(W),
	\]
	that is, $|G|=\dfrac{|\widetilde{H}|\dim(V)}{\dim(W)}$. By the induction hypothesis, $\dim(W)\mid |\widetilde{H}|$, and the result follows.
\end{proof}

\begin{Corollary}
	Suppose either $\operatorname{char}(K)=0$ or $\operatorname{char}(K)\nmid|G|$. Let $G$ be a non-abelian finite group of order $p^{n}$, where $p>1$ is a prime number. Then $KG$ admits a regular grading by some abelian group $T$.
\end{Corollary}

\begin{proof}
	Let $V$ be an irreducible $KG$-module. Since $G$ is a nilpotent group, by Lemma \ref{div.|G|} we have
	\[
	\dim(V)\mid |G|.
	\]
	   In particular, $\dim(V)$ is a power of $p$. It follows that all simple components of $KG$ are matrix algebras of the form $M_{p^{k}}(K)$, with $k\leq n$. The result now follows from Lemma \ref{p.group}.
\end{proof}	

\section{Acknowledgements}

The first author would like to thank Lucio Centrone for the kind invitation and warm hospitality during his two-month research stay at the Department of Mathematics of the University of Bari, which contributed to this work. 

The second author thanks GNSAGA.

The third author acknowledges that this article forms part of his Ph.D. thesis and was developed during his research visit at the Department of Mathematics of the University of Bari.

\end{document}